\documentclass[]{interact}

\usepackage{epstopdf}% To incorporate .eps illustrations using PDFLaTeX, etc.
\usepackage{subfigure}% Support for small, `sub' figures and tables

\usepackage[numbers,sort&compress,merge]{natbib}% Citation support using natbib.sty
\bibpunct[, ]{(}{)}{,}{n}{,}{,}% Citation support using natbib.sty
\usepackage{algorithm}
\usepackage[noend]{algpseudocode}
\makeatletter
\def\BState{\State\hskip-\ALG@thistlm}
\makeatother
\usepackage{graphicx}
\usepackage[space]{grffile}
\usepackage{latexsym}
\usepackage{textcomp}
\usepackage{tabulary}
\usepackage{booktabs,array,multirow}
\usepackage{amsfonts,amsmath,amssymb,amsthm}
\usepackage{mathrsfs}
\usepackage{csquotes}
\usepackage{blkarray}

\usepackage{xcolor}
\usepackage{lineno}
\usepackage{tabu}
\usepackage{caption}
\usepackage{graphicx}
\usepackage{multicol}
\usepackage{mathtools}
\DeclareMathOperator{\supp}{supp}
\DeclareMathOperator{\Ima}{Im}
\DeclareMathOperator{\diag}{diag}
\DeclareMathOperator{\spn}{span}

\newtheorem{theorem}{Theorem}
\newtheorem{lemma}{Lemma}
\newtheorem{proposition}{Proposition}
\newtheorem{corollary}{Corollary}

\newtheorem{definition}{Definition}
\newtheorem{example}{Example}
\newtheorem{remark}{Remark}
\numberwithin{equation}{section}

\begin{document}

%\articletype{ARTICLE TEMPLATE}

\title{Linear conjugacy of chemical kinetic systems}

\author{
Allen L. Nazareno\textsuperscript{a}, Raymond Paul L. Eclarin\textsuperscript{b}, Eduardo R. Mendoza\textsuperscript{c,d,e}, and \name{Angelyn R. Lao\textsuperscript{c}\thanks{CONTACT Angelyn R. Lao. Email: angelyn.lao@dlsu.edu.ph}}
\affil{\textsuperscript{a}Institute of Mathematical Sciences and Physics, University of the Philippines Los Ba\~{n}os, Laguna, 4031 Philippines; \textsuperscript{b}Department of Mathematics, Mariano Marcos State University, Ilocos Norte, 2906 Philippines; \textsuperscript{c}Mathematics and Statistics Department, De La Salle University, Manila, 0922 Philippines; \textsuperscript{d}Max Planck Institute of Biochemistry, 85152 Martinsried,  Germany; \textsuperscript{e}LMU Faculty of Physics, Geschwister -Scholl- Platz 1, 80539 Munich Germany}
}

\maketitle

\begin{abstract}
Two networks are said to be linearly conjugate if the solution of their dynamic equations can be transformed into each other by a positive linear transformation. The study on dynamical equivalence in chemical kinetic systems was initiated by Craciun and Pantea in 2008 and eventually led to the Johnston-Siegel Criterion for linear conjugacy (JSC). Several studies have applied Mixed Integer Linear Programming (MILP) approach to generate linear conjugates of MAK (mass action kinetic) systems, Bio-CRNs (which is a subset of hill-type kinetic systems when the network is restricted to digraphs), and PL-RDK (complex factorizable power law kinetic) systems. In this study, we present a general computational solution to construct linear conjugates of any ``rate constant-interaction function decomposable" (RID) chemical kinetic systems, wherein each of its rate function is the product of a rate constant and an interaction function. We generate an extension of the JSC to the complex factorizable (CF) subset of RID kinetic systems and  show that any non-complex factorizable (NF) RID kinetic system can be dynamically equivalent to a CF system via transformation. We show that linear conjugacy can be generated for any RID kinetic systems by applying the JSC to any NF kinetic system that are transformed to CF kinetic system.
\end{abstract}


\section{Introduction}

This paper presents a general computational solution to the problem of constructing linear conjugates of a chemical reaction network where each rate function is the product of a rate constant and an interaction function. We denote such a chemical kinetic system as a ``rate constant-interaction function decomposable" (RID) kinetic system. Nearly all systems studied in Chemical Reaction Network Theory (CRNT) are RID kinetic systems, but recently ``��variable $k$"�� systems have been introduced in \cite{csl:13}. Furthermore, various kinetics such as weakly monotonic ones, are not explicitly required to have this form. Our approach is based on two new results:  
\begin{enumerate}
\item The extension of the Johnston-Siegel Criterion for linear conjugacy (JSC) to the complex factorizable (CF) subset of RID kinetic systems, i.e., those whose interaction map $I_K: \Omega \rightarrow \mathbb{R}^\mathscr{R}$ factorizes via the space of complexes $\mathbb{R}^\mathscr{C}$:  $I_K = I_k\circ \psi_K$ with $\psi_K:  \Omega \rightarrow \mathbb{R}^\mathscr{C}$, with $\mathbb{R}_>^{\mathscr{S}}\subset \Omega \subset \mathbb{R}_\ge^\mathscr{S}$, as factor map and $I_k = diag(k)\circ \rho'$   with  $\rho': \mathbb{R}^\mathscr{C} \rightarrow \mathbb{R}^\mathscr{R}$ assigning the value at a reactant complex to all its reactions (Theorem~\ref{thm:4}).

\item The dynamic equivalence of any non-complex factorizable (NF) RID kinetic system to a CF-system (Theorem~\ref{thm:1}).
\end{enumerate}

An essential ingredient of the proofs of both results is the coincidence of the interaction maps of the kinetics considered. In the JSC extension (Theorem 4), the equality of the factor maps $\psi_K = \psi_K'$ (which is clearly equivalent to that of the interaction maps) is assumed. The CF-RM (Complex Factorizable by Reactant Multiples) transformation used to provide the dynamical equivalence in Theorem 1 is based on the concept of CF subsets of a reactant complex, which are defined as subsets of its reactions with the same interaction map.  Determining the equality of functions (with infinite definition domains) may be computationally challenging, depending on their complexity and expression format. However, for a large subset of RID kinetic systems, which we call RID systems with interaction parameter maps (and denote with RIPK), the computational feasibility is ensured. Such systems are characterized by the existence of a map $P_K: \mathscr{R} \rightarrow \mathbb{R}^p$ such that $P_K = P_{K'}$ implies $I_K = I_{K'}$. The exponent $p$ is typically (but not always) a multiple of $m$ (= number of species), and written as an $r \times p$ matrix. The interaction parameter map is easily seen as a generalization of the kinetic order matrix $F$ of power law kinetic systems.\\

Most RID kinetic systems, whose rate functions are specified explicitly, have interaction parameter maps, including all biochemical formalisms introduced to date. We discuss how the mixed integer linear programming (MILP) algorithms originally introduced for mass action kinetics (MAK)  systems can be extended to RIP kinetic systems. We illustrate this and other results of the paper with an example of Hill-type kinetics (HTK), which was originally introduced as ``Saturation Cooperativity Formalism" (SC Formalism) in \cite{csl:6}.\\

The foundations for the study of dynamic equivalence in chemical kinetic systems were laid in the paper of Craciun and Pantea \cite{csl:17}. Important contributions to the theory in a more general context were previously provided by G. Farkas in \cite{farkas}. The MILP-based computational approach to dynamic equivalence of MAK systems was pioneered by the group led by G. Szederk$\acute{e}$nyi and K. Hangos in Budapest, with further contributions from the lab of J. Banga in Vigo. Independently, M. Johnston and D. Siegel initiated the study of linear conjugacy, which led to the JSC for MAK systems. The three groups then collaborated in extending the MILP approach to linear conjugacy (a detailed discussion of the work up to 2013 can be found in \cite{csl:19}). Further developments included the extension to ``Bio-CRNs" (whose rate functions are mass action functions divided by positive polynomials in the species variables) by G$\acute{a}$bor et al. \cite{csl:18} and to complex factorizable power law kinetic systems (denoted by PL-RDK) by Cortez et al. \cite{csl:8}.\\

The paper is organized as follows: Section 2 collects the fundamentals of chemical reaction networks and kinetic systems required for the later sections. The central concept of ``CF subsets of a reactant complex" and the method based on it are introduced in Section 3. The first main result (Theorem 1) is proved using the transformation. A Subspace Coincidence Theorem for the kinetic and stoichiometric subspaces (KSSC) of NF kinetic systems further illustrates the usefulness of CF-RM. Section 4 formulates the linear conjugacy problem for RID kinetic systems and extends the Johnston-Siegel Criterion (JSC) for linear conjugacy to complex factorizable RID systems. This is combined with the CF-RM method to provide the general computational solution to construct linear conjugates of any RID system. A running example (Examples 2 - 4), in Sections 3 and 4, further demonstrates the usefulness of the computational solution by deriving the existence of complex balanced equilibria of an NF power law kinetic system through construction of a weakly reversible, deficiency one PL-TIK system which is linear conjugate to the CF-RM transform. Section 5 focusses on the large subset of RID systems which have interaction parameter maps, for which the computational solution is always feasible. Details of the MILP-based algorithm are provided in Section 6. Section 7 illustrates the results of the paper using a reference system introduced in \cite{csl:6}. Conclusions and an outlook constitute Section 8. Tables of acronyms and frequently used symbols are provided in Supplementary Materials.

%http://www.aimspress.com \cite{authour1}. 

%%%%%%%%%%%%%%%%%%%%%%%%
%%%%%%%%%%%%%%%%%%%%%%%
%%%%%%%%%%%%%%%%%%%%%%%
\section{Materials and method}
%\section{Fundamentals of chemical reaction networks and kinetic systems}

{\label{sec:2}}
We recall the necessary concepts of chemical reaction networks and the mathematical notation used throughout the paper adopted from the papers \cite{csl:8,csl:9,csl:3,csl:20}. 

\subsection{Fundamentals of chemical reaction networks}

We begin with the definition of a chemical reaction network.

\begin{definition}
A \textbf{chemical reaction network} is a triple $\mathscr{N}=(\mathscr{S},\mathscr{C},\mathscr{R})$ of three non-empty finite sets:
\begin{enumerate}
\item A set \textbf{species} $\mathscr{S}$,
\item A set $\mathscr{C}$ of \textbf{complexes}, which are non-negative integer linear combinations of the species, and
\item A set $\mathscr{R} \subseteq \mathscr{C} \times \mathscr{C}$ of \textbf{reactions} such that
\begin{itemize}
\item $(y,y) \notin \mathscr{R}$ for all $y \in \mathscr{C}$, and
\item for each $y \in \mathscr{C}$, there exists a $y' \in \mathscr{C}$ such that $(y,y') \in \mathscr{R}$ or $(y',y) \in \mathscr{R}.$
\end{itemize}
\end{enumerate}

\end{definition}

We denote with $m$ the number of species, $n$ the number of complexes and $r$ the number of reactions in a CRN.

A complex is called \textbf{monospecies} if it consists of only one species, i.e., of the form $kX_i$, $k$ a non-negative integer and $X_i$ a species. It is called \textbf{monomolecular} if $k=1$, and is identified
with the \textbf{zero complex} for $k=0$. A zero complex represents the \enquote{outside} of the system studied, from which chemicals can flow into the system at a constant rate and to which they can flow out at a linear rate (proportional to the concentration of the species). In biological systems, the \enquote{outside} also stands for the degradation of a species. 

A chemical reaction network $(\mathscr{S}, \mathscr{C}, \mathscr{R})$ gives rise to a digraph with complexes as vertices and reactions as arcs. However, the digraph determines the triple uniquely only if an additional property is considered in the definition: $\mathscr{S}=\bigcup$\{\ supp $y$ for $y \in \mathscr{C}\}$, i.e., each species appears in at least one complex. With this additional property, a CRN can be equivalently defined as follows.
 
\begin{definition}
A \textbf{chemical reaction network} is a digraph $(\mathscr{C}, \mathscr{R})$ where each vertex has positive degree and stoichiometry, i.e., there is a finite set $\mathscr{S}$ (whose elements are called \textbf{species}) such that $\mathscr{C}$ is a subset of $\mathbb{Z}^{\mathscr{S}}_{\geq}.$ Each vertex is called a \textbf{complex} and its coordinates in $\mathbb{Z}^{\mathscr{S}}_{\geq}$ are called \textbf{stoichiometric coefficients}. The arcs are called \textbf{reactions}.
\end{definition}

Two useful maps are associated with each reaction:

\begin{definition}
The \textbf{reactant map} $\rho: \mathscr{R} \rightarrow \mathscr{C}$ maps a reaction to its reactant complex while the \textbf{product map} $\pi: \mathscr{R} \rightarrow \mathscr{C}$ maps it to its product complex. We denote $|~\rho (\mathscr{\pi})~|$ with $n_r$, i.e., the number of reactant complexes. 
\end{definition}

Connectivity concepts in Digraph Theory apply to CRNs, but have slightly differing names. A connected component is traditionally called a \textbf{linkage class}, denoted by $\mathscr{L}$, in CRNT. A subset of a linkage class where any two elements are connected by a directed path in each direction is known as a \textbf{strong linkage class}. If there is no reaction from a complex in the strong linkage class to a complex outside the same strong linkage class, then we have a \textbf{terminal strong linkage class}. We denote the number of linkage classes with $l$, that of the strong linkage classes with $sl$ and that of terminal strong linkage classes with $t$. Clearly, $sl \geq t \geq l.$

Many features of CRNs can be examined by working in terms of finite dimensional spaces $\mathbb{R}^{\mathscr{S}}, \mathbb{R}^{\mathscr{C}}, and \mathbb{R}^{\mathscr{R}},$ which are referred to as species space, complex space, and reaction space, respectively. We can view a complex $y \in \mathscr{C}$ as a vector in $\mathbb{R}^{\mathscr{C}}$ (called \textit{complex vector}) by writing $y = \sum _{x \in \mathscr{S}} y_x x$, where $y_x$ is the stoichiometric coefficient of species $x$.

\begin{definition}
The \textbf{reaction vectors} of a CRN $(\mathscr{S}, \mathscr{C}, \mathscr{R})$ are the members of the set $\displaystyle{\{y'-y \in \mathbb{R}^{\mathscr{S}}~|~(y,y') \in \mathscr{R}\}}.$ The \textbf{stoichiometric subspace} $S$ of the CRN is the linear subspace of $\mathbb{R}^{\mathscr{S}}$ defined by 
	$$S: span \{y'-y \in \mathbb{R}^{\mathscr{S}}~|~(y,y') \in \mathscr{R}\}.$$
The \textbf{rank} of the CRN, $s$, is defined as $s=dim~S.$
\end{definition}

\begin{definition}
The \textbf{incidence map} $I_a: \mathbb{R}^{\mathscr{R}} \rightarrow \mathbb{R}^{\mathscr{C}}$ is defined as follows. For $f: \mathscr{R} \rightarrow \mathbb{R}$, then $\displaystyle{I_a(f)(v) = - f(a)}$ and $f(a)$ if $v = \rho(a)$ and $v = \pi(a)$, respectively, and are $0$ otherwise.
\end{definition}

\noindent Equivalently, it maps the basis vector $\omega_a$ to  $\omega_{v'} -  \omega_v$ if $a: v \rightarrow v'$. It is clearly a linear map, and its matrix representation (with respect to the standard bases $\omega_a$, $\omega_{v}$) is called the \textbf{incidence matrix}, which can be described as 
\begin{center}
\[
 (I_a)_{i,j} = 
  \begin{cases} 
   -1 & \text{if } \rho(a_j) = v_i, \\
   1       & \text{if } \pi(a_j) = v_i,\\
   0		& \text{otherwise}.
  \end{cases}
\]
\end{center}
%\noindent Note that in most digraph theory books, the incidence matrix is set as $-I_a$.\\
%An important result of digraph theory regarding the incidence matrix is the following:

%\begin{proposition}
%\label{prop:IncidenceMatrix02}
Let $I$ be the incidence matrix of the directed graph $D = (V, E)$. Then rank $I = n -l$, where $l$ is the number of connected components of $D$. A non-negative integer, called the deficiency, can be associated to each CRN. This number has been the center of many studies in CRNT due to its relevance in the dynamic behavior of the system. The \textbf{deficiency} of a CRN is the integer $\delta = n-l-s$.  The \textbf{reactant subspace} $R$ is the linear space in $\mathbb{R}^\mathscr{S}$ generated by the reactant complexes. Its dimension, denoted by $q$, is called the \textbf{reactant rank} of the network. Meanwhile, the \textbf{reactant deficiency} $\delta_p$ is the difference between the number of reactant complexes and the reactant rank, i.e., $\delta_p = n_r -q$.

%%%%%%%%%%%%%%%%%%%%%%%%%%%%%%%%
%%%%%%%%%%%%%%%%%%%%%%%%%%%%%%%%%
%%%%%%%%%%%%%%%%%%%%%%%%%%%%%%%%% 
\subsection{Fundamentals of chemical kinetic systems}

We now introduce the fundamentals of chemical kinetic systems. We begin with the general definitions of kinetics from  \cite{csl:22}:

\begin{definition}
A \textbf{kinetics} for a CRN $\mathscr{N}=(\mathscr{S}, \mathscr{C}, \mathscr{R})$ is an assignment of a rate function $K_j: \Omega_K \rightarrow \mathbb{R}_\geq$ to each reaction $r_j \in \mathscr{R}$, where $\Omega_K$ is a set such that $\mathbb{R}^\mathscr{S}_> \subseteq \Omega_K \subseteq \mathbb{R}^\mathscr{S}_\geq$, $c\wedge d \in \Omega_K$ whenever $c,d \in \Omega_K,$ and 
$$K_j(c) \geq 0, \quad \forall c \in \Omega_K.$$
A kinetics for a network $\mathscr{N}$ is denoted by $\displaystyle{K=(K_1,K_2,...,K_r):\Omega_K \to {\mathbb{R}}^{\mathscr{R}}_{\geq}}$. A \textbf{chemical kinetics} is a kinetics $K$ satisfying the positivity condition: for each reaction $r_j:y\rightarrow y', K_j(c)>0$ iff $\supp y\subset\supp c$. The pair $(\mathscr{N}, {K})$ is called the \textbf{chemical kinetic system} (CKS).
\end{definition}

%\noindent 
In the definition, $c \wedge d$ is the bivector of $c$ and $d$ in the exterior algebra of $\mathbb{R}^\mathscr{S}$. Once a kinetics is associated with a CRN, we can determine the rate at which the concentration of each species evolves at composition $c$.\\% We add the definition relevant to our context:

%\begin{definition}
%A \textbf{chemical kinetics} is a kinetics $K$ satisfying the positivity condition: for each reaction $r_j:y\rightarrow y', K_j(c)>0$ iff $\supp y\subset\supp c$.
%\end{definition}
%\noindent Once a kinetics is associated with a CRN, we can determine the rate at which the concentration of each species evolves at composition $c$.

Power-law kinetics is defined by an $ r \times m$ matrix $F =[F_{ij}],$ called the \textbf{kinetic order matrix}, and vector $k \in \mathbb{R}^\mathscr{R}$, called the \textbf{rate vector}. In power-law formalism, the kinetic orders of the species concentrations are real numbers.

\begin{definition}
A kinetics $K: \mathbb{R}^\mathscr{S}_> \rightarrow \mathbb{R}^\mathscr{R}$ is a \textbf{power-law kinetics} (PLK) if
$$K_i (x) = k_ix^{F_i} \quad \forall i = 1, ... , r$$
with $k_i \in \mathbb{R}_>$ and $ F_{ij} \in \mathbb{R}.$
\end{definition}

\begin{definition}
A chemical kinetics $K:\Omega \rightarrow \mathbb{R}^\mathscr{R}_\ge$ is \textbf{complex factorizable (CF)} if there is $k\in \mathbb{R}^\mathscr{R}_>$ and a mapping $\psi_K: \Omega \rightarrow \mathbb{R}^{\mathscr{C}}$ such that $K=I_k\circ \psi_K$, where $I_k$ is the $\bf{k-}$\textbf{interaction map} defined by $I_k: \mathbb{R}^\mathscr{C} \rightarrow \mathbb{R}^\mathscr{R}$. The set of complex factorizable kinetics is denoted as $\mathscr{CFK(N)}$. 
\end{definition}

It can be deduced from the definition that if a chemical kinetics $K$ is complex factorizable, then its \textbf{complex formation rate function} $g=A_k\circ \psi_K$ and its \textbf{species formation rate function (SFRF)} $f=Y\circ A_k\circ \psi_K$. The $f(x) = \frac{dx}{dt}$ is the ODE or dynamical system of the {CKS}. A zero of $f$ is an element $c$ of $\mathbb{R}^\mathscr{S}$ such that $f(c)=0$. A zero of $f$ is called an equilibrium (or steady state) of the ODE system. The SFRF contains three maps:  map of complexes, Laplacian map, and factor map.

\begin{definition} The \textbf{map of complexes} $Y: \mathbb{R}^\mathscr{C} \rightarrow \mathbb{R}^\mathscr{S}$ is defined by its values on the standard basis $\{\omega_y\}$ , $y$ a non-zero complex: $Y(\omega_y) = y$ and extending it linearly to all elements of $\mathbb{R}^\mathscr{C}$.  Its matrix, denoted with $Y$ (called the \textbf{matrix of complexes}), is an $m \times n$ matrix, its rows indexed by the species and its column by the complexes, with $y_{ij}$ being the stoichiometric coefficient of the $j^{th}$ complex in the $i^{th}$ species. In other words, the columns are the complexes written as column vectors.
\end{definition}

\begin{definition} The linear transformation $A_k: \mathbb{R}^\mathscr{C} \rightarrow \mathbb{R}^\mathscr{C}$ called \textbf{Laplacian map} is the mapping defined by $A_kx:=\sum_{(i,j)\in \mathscr{R}} k_{ij}x_i(\omega_j-\omega_i)$, where $x_i$ refers to the $i^{\text{th}}$ component of $x\in \mathbb{R}^\mathscr{C}$ relative to the standard basis. Its matrix representation is the $n\times n$ matrix such that
 $$(A_k)_{ij} = 
  \begin{cases} 
   k_{ji} & \text{if } i\neq j, \\
   k_{jj}-\sum_{i'=1}^n k_{ji'} & \text{if } i= j.
  \end{cases}
$$
\noindent The label $k_{ji}$  is called the rate constant and is associated to the reaction $(j,i)\in \mathscr{R}$.
\end{definition}

\begin{definition} The \textbf{factor map} $\psi_K: \Omega \rightarrow \mathbb{R}^\mathscr{C}$ is defined as
$$
 (\psi_K)_{c}(x) = 
  \begin{cases} 
   (x^F)_i & \text{if } c \text{ is a reactant complex of a reaction } i, \\
   1 &otherwise.
  \end{cases}
$$
\end{definition}

\begin{definition}
A \textbf{positive equilibrium} or steady state $x$ is an element of $\mathbb{R}^\mathscr{S}_>$ for which $f(x)=0$. The set of positive equilibria of a chemical kinetic system is denoted by $E_+(\mathscr{N} , {K})$.
\end{definition}

Two networks are said to be \textbf{linearly conjugate} if the solutions of their dynamic equations can be transformed into each other by a positive linear transformation \cite{csl:20,csl:21}.

\begin{definition}
Let $\Phi(x_0, t)$ and $\psi (x_0, t)$ be flows associated to kinetic systems $M$ and $M'$ respectively. $M$ and $M'$ are said to be \textbf{linearly conjugate} if there exists a bijective linear mapping $h : \mathbb{R}^n_{>0} \rightarrow \mathbb{R}^n_{>0}$ such that $h(\Phi(x_0, t)) =  \psi(h(x_0, t))$ for all $x_0 \in \mathbb{R}^n_{>0}$.
\end{definition}

\begin{remark}
In \cite{csl:16a}, it is shown that the bijection h in the previous definition corresponds to multiplication with a diagonal matrix with positive diagonal entries. The diagonal entries form the conjugacy vector $c$. More precisely, if $N$, $N'$ are the stoichiometric matrices and $K$, $K'$ are the  kinetics of the systems $M$ and $M'$ respecively, then they are linearly conjugate if and only if $NK = \diag (c)N'K'$. 
\end{remark}

Linear conjugacy is  a generalization of the concept of dynamical equivalence. 

\begin{definition}
Two kinetic systems are dynamically equivalent if the conjugacy vector $c = (1,\cdots ,1)$, i.e., if $NK = N'K'$.
\end{definition}

In relation to linear conjugacy, if the mapping $h$ is trivial, $M$ and   $M'$  are said to be \textbf{dynamically equivalent} \cite{csl:8} .

%%%%%%%%%%%%%%%%%%%%%%%%%%%%%%%%%%%%%%
\subsection{Rate constant-Interaction map Decomposable (RID) kinetics}
To date, nearly all chemical kinetics studied in CRNT have constant rates, i.e. for each reaction $r$, the kinetic function $K_r: \Omega_K \rightarrow \mathbb{R}^\mathscr{R}$ can be written in the form  $K_r(x) = k_r I_{K,r}(x)$, with a positive real number $k_r$ (called a rate constant) and an interaction map $I_{K,r}$. Recently however, G. Craciun and collaborators \cite{csl:13,csl:12} have introduced \textbf{variable $k$ systems}, where the rates may vary between an upper and lower bound. Furthermore, there are kinetics sets such as the weakly monotonic kinetics studied in \cite{csl:11} or the span surjective kinetics introduced in \cite{csl:1} which do not explicitly require constant rates. The \textbf{fractal kinetics} studied primarily by physical chemists, e.g. Brouers \cite{csl:12} have rate values given by a function of exponential type. In view of this, we introduce the term \textbf{Rate constant-Interaction map Decomposable} (RID) kinetics for all chemical kinetics with constant rates and denote the set with RIDK.

In \cite{csl:9} (see also \cite{csl:1}), we introduce a special subset of $\mathscr{CFK(N)}$, which is the set of power law kinetics with reactant-determined kinetic orders, denoted by $\mathscr{PL-RDK (N)}$. A PLK system has a \textbf{reactant-determined kinetic orders} (of type PL-RDK) if for any two reactions $i,j$ with identical reactant complexes, the corresponding rows of kinetic orders in $V$ are identical, i.e., $v_{ik}=v_{jk}$ for $k=1,2,...,m$.

We note also in  \cite{csl:1} that $\mathscr{PL-RDK(N)}$ includes mass action kinetics (MAK) and coincides with the set of GMAK systems recently introduced by \cite{csl:10} if the vertices map $y: \mathscr{C} \rightarrow R^m$ of the GMAK system  is injective. They also constitute the subset of power law systems for which various authors claimed that their results ``hold for the complexes with real coefficients" are valid. 

Another important property of a complex factorizable kinetics is ``factor span surjectivity":\\

\begin{definition}
Let $f:V\rightarrow W$ be a map between finite dimensional vector spaces $V$ and $W$. $f$ is \textbf{span surjective} if and only if $\spn(\Ima f)=W$.
\end{definition}

\noindent In \cite{csl:1}, it is shown that $f$ is span surjective if and only if its coordination functions are linearly independent.

\begin{definition}
A complex factorizable kinetics $K$ is \textbf{factor span surjective} if its factor map $\psi_K$ is span surjective. $\mathscr{FSK(N)}$ denotes the set of factor span surjective kinetics on a network $\mathscr{N}$. 
\end{definition}

We characterized in \cite{csl:1} a factor span surjective PL-RDK system. 

\begin{proposition}
A PL-RDK system is factor span surjective if and only if no rows corresponding in the kinetics order matrix $F$ corresponding to different reactant complexes coincide (i.e. $\rho(r)\ne \rho(r')\Rightarrow F_r\ne F_{r'}$). 
\end{proposition}

We recall the definition of the $m \times n$ matrix $\tilde{Y}$ from \cite{csl:10}: for a reactant complex, the column of $\tilde{Y}$ is the transpose of the kinetic order matrix row of the complex' reaction, otherwise (i.e. for a terminal point), the column is 0.\\

The $\bf{T}-$\textbf{matrix} of a PL-RDK system is formed by truncating away the columns of the terminal points in $\tilde{Y}$, obtaining an $m \times n_r$ matrix. The corresponding linear map $T : \mathbb{R}^{\rho(\mathscr{R})} \rightarrow \mathbb{R}^{\mathscr{R}}$ maps $\omega_\rho( r )$ to $(F_r )^T$ . The subspace $\tilde{R} : = \text{ Im }T =  \big \langle (F_r )^T \big \rangle$   is called the kinetic reactant subspace and $\tilde{q} = \dim \tilde{R}$ is called the kinetic reactant rank of the system.\\

Let $e^1,e^2,...,e^\ell \in \left\{ 0,1\right\}^{n}$ be the characteristic vectors of the sets $\mathscr{C}^{1}$,$\mathscr{C}^{2}$,...,$\mathscr{C}^{\ell}$, respectively, where $\mathscr{C}^{i}$ is the set of complexes in linkage class $\mathscr{L}^{i}$. That is, for all $j \in \mathscr{C}$ and  $i = 1,\dots,\ell$, we have $e^i_j = 1$ if $j \in \mathscr{C}^i$, and 0 otherwise. Let $L = \left[ e^1,e^2,...,e^\ell \right]$. Define the $\bf{\hat{T}}-$\textbf{matrix}, an $(m+\ell)\times n_{r}$ block matrix,  by $$\hat{T}=\left[ 
\begin{array}[center]{c} T \\
L_{pr}^{\top} \\
\end{array} \right],$$ 
where $L_{pr}$ is the truncated matrix $L$ (i.e., non-reactant rows are left out).\\

If the non-inflow columns (i.e., columns of the complexes associated to non-inflow reactions) of ${T}-${matrix} corresponding to each linkage class are linearly independent and if its column rank is maximal, then the chemical kinetics is said to be $\bf{\hat{T}}-$\textbf{rank maximal} (to type \textbf{PL-TIK}).

\section{CF Transformation of NF Kinetics}

The CF-RM (Complex Factorization by Reactant Multiples) method developed from a proposal by C. Pantea in December 2017 for such a transformation for power law kinetics. The key idea is, at  an NF branching point, i.e. a complex which is the reactant of reactions (called its branching reactions) with non-proportional interaction maps, to transform reactions by introducing new reactants while conserving the reaction vectors, thus leaving the stoichiometric subspace invariant. CF-RM refines the approach by ensuring that the reactant subspace also remains invariant and that a minimum number of reactions is transformed.  The essential underlying concept of CF-RM is that of CF-subsets of the set of reactions of a reactant complex.  The concept is also the basis for the construction of CF-decompositions of a RID kinetic system. 

\subsection{CF-subsets of the reaction set of a reactant complex}

For a reactant complex $y$ of a network $\mathscr{N}$, $\mathscr{R}(y)$ denotes its set of (branching) reactions, i.e., $\rho^{-1}(y)$ where $\rho:\mathscr{R} \rightarrow \mathscr{C}$ is the reactant map.  The $n_r$ reaction sets $\mathscr{R}(y)$ of reactant complexes partition the set of reactions $\mathscr{R}$ and hence induce a decomposition of $\mathscr{N}$. 

\begin{definition}  
Two reactions $r$, $r' \in \mathscr{R}(y)$ are \textbf{CF-equivalent for K} if their interaction functions coincide, i.e., $I_{K,r} = I_{K,r'}$ or, equivalently, if their kinetic functions $K_r$ and $K_r'$ are proportional (by a positive constant). The equivalence classes are the \textbf{CF-subsets} (for $K$) of the reactant complex $y$. 
\end{definition}

\begin{definition}
If $N_R(y)$ is the number of CF-subsets of $y$, then $\displaystyle{1 \le N_R(y) \le \mid \rho^{-1}(y)\mid}$. The reactant complex is a \textbf{CF-node} if $N_R(y)= 1$, and an NF-node otherwise. It is a \textbf{maximally NF-node} if $\displaystyle{N_R(y) = \mid\rho^{-1}(y)\mid> 1}$.
\end{definition}

\begin{definition}
The number \textbf{N$_\textbf{R}$}  of CF subsets of a CRN is the sum of $N_R(y)$ over all reactant complexes. 
\end{definition}

Clearly, $N_R  \ge n_r$ and the kinetics $K$ is CF if and only if $N_R = n_r$, or equivalently all reactant complexes are CF-nodes for $K$.

\begin{example}
 For a power law kinetic system, the CF-subsets of a reactant complex are the subsets of branching reactions with identical rows in the kinetic order matrix. To show this, we recall that the interaction map of a PLK system is $= x^F$ and hence the claim is $x^{l(r)} = x^{l(r')} \Rightarrow l(r)  = l(r')$  . The ``$\leq$" is evident, for the converse, let $e_i$ be the positive vector with e (the exponential number) as its ith coordinate and 1's otherwise. Since  $\log x^{l(r)} =  l(r)\log x$, the value of the $\log$ at $e_i =$ the ith kinetic order, which proves the claim. 
\end{example}

\begin{example} (Running Example - Part 1) In \cite{FMRL2019}, a power law kinetic system for R. Schmitz's pre-industrial carbon cycle model was introduced. The system (depicted in Figure~\ref{fig:smpriga}) with 6 complexes (representing carbon pools) and 13 reactions (indicating mass transfer) is weakly reversible and has zero deficiency.
\begin{figure}[H]
    \centering
    \includegraphics[width=0.35\textwidth]{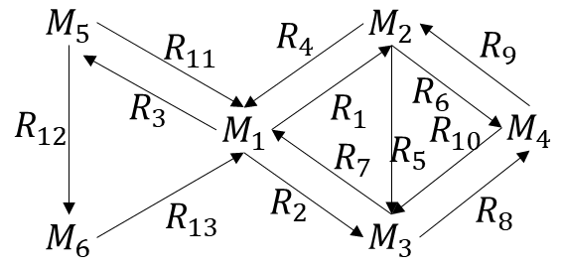}
\caption{CRN corresponding to the R. Schimtz's pre-industrial carbon cycle model \cite{FMRL2019}.}
\label{fig:smpriga}
\end{figure}

\noindent The system's kinetic order matrix is given by:                
$$M=\begin{blockarray}{lccccccl}
& M_1 & M2 & M3 & M4 & M5 & M6 \\
\begin{block}{c(cccccc)c}
  		R_1 & 1 & 0 & 0 & 0 & 0 & 0 &0.0931 \\
                  R_2 & 1 & 0 & 0 & 0 & 0 & 0 &0.0311 \\
                  R_3 & 0.36 & 0 & 0 & 0 & 0 & 0 &10.08896 \\
                  R_4 & 0 & 9.4 & 0 & 0 & 0 & 0 &0.7 \\
                  R_5 & 0 & 1 & 0 & 0 & 0 & 0 &0.0781 \\
                  R_6 & 0 & 1 & 0 & 0 & 0 & 0 &0.0164 \\
                  R_7 & 0 & 0 & 10.2 & 0 & 0 & 0 &0.2 \\
                  R_8 & 0 & 0 & 1 & 0 & 0 & 0 &0.714 \\
                  R_9 & 0 & 0 & 0 & 1 & 0 & 0 &0.0164 \\
                  R_{10} & 0 & 0 & 0 & 1 & 0 & 0 &0.00114 \\
                  R_{11} & 0 & 0 & 0 & 0 & 1 & 0 &0.0862 \\
                  R_{12} & 0 & 0 & 0 & 0 & 1 & 0 &0.0862 \\
                  R_{13} & 0 & 0 & 0 & 0 & 0 & 1 &0.0333 \\
\end{block}
\end{blockarray}.$$

\noindent The column to the right of F lists the rate constants of the corresponding reactions. The kinetic order matrix reveals that the system has 3 NF nodes (reactant complexes): $M_1$, $M_2$ and $M_3$. The following table lists their CF-subsets:
\begin{center}
  \begin{tabular}{ | c | c | c |}
    \hline
    NF node & Reaction set & CF-subsets \\ \hline
    $M_1$ & $\{R_1, R_2, R_3\}$ & $\{R_1, R_2\},\{R_3\}$ \\ \hline
    $M_2$ & $\{R_4, R_5\}$ & $\{R_4\},\{R_5\}$ \\ \hline
    $M_3$ & $\{R_6, R_7\}$ & $\{R_6\},\{R_7\}$ \\ 
    \hline
  \end{tabular}
\end{center}

\noindent Hence, $\mathscr{R}$ is partitioned into 9 CF-subsets, i.e., $N_R = 9$.
 
\end{example}

We also note that since the CF-subsets of a reactant complex partition its reaction set, and the reaction sets of reactant complexes partition the set of reactions, that the CF-subsets determine a decomposition.

We recall from \cite{JOSH2013} that a subset $\mathscr{R}'$ of $\mathscr{R}$ defines a subnetwork  $\mathscr{N}'= (\mathscr{S}', \mathscr{C}', \mathscr{R}')$ with $\mathscr{C}'$ consisting of the complexes occurring in reaction of $\mathscr{R}'$ and $\mathscr{S}'$ consisting of the species occurring in complexes in $\mathscr{C}'$. A CRN decomposition $\mathscr{N} = \mathscr{N}_1 \cup ... \cup \mathscr{N}_k$  consists of the subnetworks $\{ \mathscr{N}_i \}$ induced by a partition $\{ \mathscr{R}_i \}$ of $\mathscr{R}$. We use the model presented in \cite{FMRL2018} to illustrate the concepts introduced above.

\begin{definition}
The CF-subsets of a RID kinetic system partition the reaction set and induce the \textbf{CFS decomposition} of the system.
\end{definition}

The CFS decomposition consists of $N_R$ subnetworks, whereby $n_r \le N_R \le r$.

\subsection{CFM decompositions of a RID kinetic system }
In this Section, we introduce useful coarsenings of the CFS-decomposition of a RID kinetic system.\\

For each NF node $y$, we choose an ordering of its CF-subsets $\mathscr{R}_1(y)$, $\mathscr{R}_2 (y)$,..., $\mathscr{R}_{NR}(y)$ according to decreasing number of reactions in the CF-subset.  We define $\mathscr{R}_i := \cup_{y\in \rho(R)} \mathscr{R}_i(y)$ where $i = 1,...,\max_{y\in \rho(R)}  N_R(y)$  and $\mathscr{R}_i(y') = \phi$ if $N_R(y') < i$. \\

We can now introduce the concept of a maximal CF-subsystem (CFM) of a RID kinetic system:

\begin{definition}
A \textbf{maximal CF-subsystem} $(\mathscr{N}_{mcf}, {K})$ of a RID kinetic system $(\mathscr{N}, {K})$ is induced the union of the reaction sets of all CF-nodes and a CF-subset with the maximal number of reactions from each NF-node, i.e., the union $\mathscr{R}_{mcf}$ of $\{ \rho^{-1}(y)\mid y\ is \ CF-node\}$ and $\mathscr{R}_1$.
\end{definition}

Clearly, there may be several maximal CF-subsystems in a RID kinetic system, but the number of reactions in each of them is the same, and we denote this with $r_{mcf}$.  Note that since $\mid \mathscr{R}_i\mid \ge \mid \mathscr{R}_j\mid$ if $i < j$, then $r_{mcf}\ge \mid \mathscr{R}_i\mid$ for all $i$.
 
\begin{definition}
A \textbf{CFM decomposition} is induced by the reaction set partition $\{\mathscr{R}_{mcf}, \mathscr{R}_2,..., \mathscr{R}_k\}$, with $k = \max_{y\in \rho(R)}  N_R(y)$. 
\end{definition}

A CFM-decomposition is clearly a coarsening of the CFS-decomposition. It is the decomposition into CF-subsystems with the least number of subnetworks. 
\subsection{CF-RM Transformation of an NF kinetic system: the generic case}
We first introduce the concept of a CF-transformation of an NF kinetic system:

\begin{definition}
 A CF kinetic system $(\mathscr{N}^*, K^*)$ is a \textbf{CF-transform} of an NF system $(\mathscr{N}, {K})$, where $\mathscr{N} = (\mathscr{S}, \mathscr{C}, \mathscr{R})$, $\mathscr{N}^{*} = (\mathscr{S}^{*}, \mathscr{C}^{*}, \mathscr{R}^{*})$ and $N$, $N^{*}$ as their respective stoichiometric matrices, if and only if  $\mathscr{S}^{*} = \mathscr{S}$, $N^{*} = N$, and $K^{*} = K$. 
\end{definition}

$N^{*}K^{*} = NK$ implies that a CF-transform is dynamically equivalent to the original NF system. Moreover, the stoichiometric subspaces coincide, i.e., $S^* = S$.

Our first main result is the following Theorem:

\begin{theorem}\label{thm:1}
Any NF system $(\mathscr{N}, {K})$ is dynamically equivalent to a CF system $(\mathscr{N}^{*} , {K}^{*})$ via a CF-transformation.
\end{theorem}
\begin{proof} We construct the CF-transformation nodewise, i.e., we specify how to transform each NF-node $y$ into $N_R(y)$ CF-nodes. Let $\mathscr{R}_1(y),..., \mathscr{R}_k(y)$ (where $k = N_R(y)$) be the CF subsets of $y$. We leave $\mathscr{R}_1(y)$ unchanged. We choose a complex $y_2$ such that $y + y_2$ is not contained in $\rho(\mathscr{R})$. All reactions in $\mathscr{R}_2(y)$ are transformed ``catalytically", i.e., $r_i: y \rightarrow z_i$ is replaced by $r_i'´: y + y_2 \rightarrow z_i + y_2$. The reaction vector is unchanged. For the reactions in $\mathscr{R}_3(y)$, choose a complex $y_3$ such that $y + y_3$ is not in $\rho(\mathscr{R}) \cup \{ y + y_2\}$ and proceed as in $\mathscr{R}_2(y)$.  After $N_R(y) -1$ steps, we have completed the transformation for $y$. After the transformation of all NF nodes, we have a CF-transform as claimed.
\end{proof}

There is clearly a multitude of ways to carry out CF-transformations, and a good principle is to minimize the changes needed as well as keep further network components invariant under the necessary changes.  In this spirit, the specific goals of the CF-RM method are:
\begin{itemize}
\item minimize the number of reactions to be changed and
\item leave the reactant subspace invariant, i.e., $R^{*} = R$.
\end{itemize}
 
The first goal is achieved by choosing, for each NF node, a CF subset with the maximal number of reactions, as the subset to be left unchanged. The second goal is accomplished by selecting the ``catalytic" complexes used as multiples of the reactant complex (as expressed in the acronym CF-RM).

The CF-RM method proceeds as follows:
\begin{itemize}
\item Determine the reactant set $\rho(\mathscr{R})$ (see Algorithm 1 lines 1-4).
\item A CF-node is left unchanged (see Algorithm 1 lines 5-21).
\item At an NF-node, select a CF-subset with the maximal number of reactions. Note that there may be several. This CF-subset is left unchanged (this step minimizes the number of t-reactions overall and may see Algorithm 1 lines 22-29).  
\item For each of the remaining $N_R(y) - 1$ CF-subets, choose successively a multiple of $y$ which is not among the current set of reactants, i.e., those of the original networks left unchanged and the already selected new reactants. Various procedures are possible for this selection of a new reactant; the essential condition is that it is different from those in the current reactant set. After each choice, the current set must be updated.
For each Non-reactant Determined Kinetics (NDK) reactant complex $y$,  $N_R(y) - 1$ new reactants are constructed (see Algorithm 1 lines 30-37).
\item Since the last expression is also true for a CF-node%($0 = 1 – 1 = N_R(y) - 1$ are constructed)
, the total number of new reactants $= \sum (N_R(y) - 1)$ with the sum taken over all reactants. This number $= \sum N_R(y) - \sum 1 = \sum N_R(y) - n_r = N_R - n_r$.  Under CF-RM the number of CF-subsets NR of the original system is also the number of reactants of the transformed system, since the latter is equal to $n_r +  N_R - n_r = N_R$.    
\end{itemize}

%\subsection{Matlab software for CF-RM for RIP-NFK} 
\begin{algorithm}%[H]
\caption{CF-RM for RIP-NFK}\label{euclid}
\begin{algorithmic}[1]
\Procedure{INITIAL}{}
\State INPUT1: reaction set with its kinetic values
\State OUTPUT1: reactant set, denote this by $\rho(\mathcal{R})$
\State OUTPUT2: matrix $\rho'$ for the reactant map of the network (from OUTPUT1)
\EndProcedure
%%%%%%%%%%%%%%%%%
\Procedure{Identification of Branching Complexes}{}
\State INPUT2: column sum of $\rho'$ (from OUTPUT2)
\State OUTPUT3: identify the branching complexes
\If {$|\rho'(y)| > 1$} 
\State \Return complex $y$ is a branching reactant complex
\Else 
	\If {$|\rho'(y)| = 1$} 
	\State \Return complex $y$ is a non-branching reactant complex
	\Else 
		\If {$|\rho'(y)| < 1$} \Return false
		\EndIf
	\EndIf
\EndIf
\EndProcedure
%%%%%%%%%%%%%%%%%
\Procedure{Identification of RDK and NDK Complexes}{}
\State INPUT3: kinetic order of the identified branching complex (from OUTPUT3)
\State OUTPUT4: determine whether the branching complex is RDK or NDK
\If {all kinetic order associated to the identified branching complex are all equal} 
\State \Return {complex is an RDK}
\Else 
\State \Return {complex is an NDK}
\EndIf
\EndProcedure
%%%%%%%%%%%%%%%%%
\Procedure{Generate RDK subsets for every NDK complexes}{}
\State INPUT4: for every NDK node $z$ (from OUTPUT4)
\State Let $N_R$ be the number of distinct kinetic order representation for each $z$. 
\State OUTPUT5: identify $N_R$ for each $z$
\State OUTPUT6: generate the reaction set $\rho^{-1} (z )$ for each $z$
\State OUTPUT7: generate the RDK subsets, $\mathcal{R}_b$, (input from OUTPUT5-6) \\ \hspace{.3in} where
$\mathcal{R}_b(z) = \{r \in \rho^{-1}(z) | \imath(r) = b\}$ and \\ \hspace{.3in} $b$ is a distinct kinetic order value in the NDK node $z$.
\EndProcedure
%%%%%%%%%%%%%%%%%
\Procedure{CF Transformation}{}
\State OUTPUT8: Take $\max \{|\mathcal{R}_b(z )|\}$ (from OUTPUT7) 
\State note: the reactions of this RDK-subset is left unchanged.
\For{$c=1 \text{ \textbf{to} } (N_R - 1)$ } 
\State {check the reactant $a$ in $\rho(\mathcal{R})$} 
\State Let $m_a$ be the coefficient of $a$ in $\rho(\mathcal{R})$. 
%\State . \\ \hspace{.6in} 
\State OUTPUT9: transform the reactions in $\mathcal{R}_b$ as such that the new reactant is $a + m_aa = (m_a + 1)a$
\State OUTPUT10: Update $\rho(\mathcal{R})$ (from OUTPUT9)
\EndFor
\EndProcedure
\State \textbf{REPEAT} \textit{Procedure CF Transformation} (for the remaining distinct kinetic order values)
\State  \textbf{REPEAT} \textit{Procedure Generate RDK subsets for every NDK complexes} (for the remaining NDK node)
\end{algorithmic}
\end{algorithm}

\begin{remark}
If an NF system has at least one NF-node with more than 1 CF-subset with the maximal number of reactions, then several transforms can be generated, which might have some differing network properties. It is possible to define an additional procedure for which CF-subset to choose and leave unchanged.
\end{remark}

\begin{remark}
As mentioned above, various procedures can be defined to select a new reactant. One possible procedure is the following: 
\begin{itemize}
\item Determine the set of multiples of $y$ among the current reactants.
\item If the set is empty, set $m_y = 1$.
\item If the set is non-empty, determine the maximum multiple $y'=max_y y$. Set $m_y = max_y$.
\item The new reactant is $y + m_y y$.
\end{itemize}
\end{remark}

Instead of repeating the reactant set check for every CF-subset of $y$, one could further optimize by ordering the CF-subsets to be changed, doing the above for the first, and then use $y + (m_y + i - 1)y$ for the $i = 2,..., N_R-1$.

Table \ref{t1} presents the key network numbers of a CF-RM transform in equations or inequalities involving only network numbers of the original NF network.  Thus, the relationships are of predictive character. 
\begin{table}
\caption{Key network number of a CF-RM transform.}
\center
\begin{tabular}{p{6cm} p{6cm}}
\hline
Network number& Value/bounds \\ 
\hline
Number of species & $m^{*} = m$  \\ 
\hline
Number of complexes &? \\ 
\hline
Number of reactant complexes & $n_r \le n_r^* = N_R$
($N_R:= \sum \mid \iota(\rho^{-1}(y))\mid$ = total number of RDK subsets) \\ 
\hline
Number of CF-subsets & $N_R^{*} = N_R$  \\ 
\hline
 Number of reactions & $r^{*} = r$  \\ 
\hline
Number of linkage classes & $1 \le l^{*} - l^{*}_b \le (N_R - n_r) + l$
($l^{*}_b :=$ number of new linkage classes from link-breaking) \\ 
\hline
Number of terminal strong linkage classes & ? \\ 
\hline
Rank of network & $s^{*} = s $  \\ 
\hline
Reactant rank of network & $q^{*} = q $  \\ 
\hline
Deficiency of network & ?  \\ 
\hline
Reactant deficiency of network & $\delta_\rho^{*} = \delta_\rho + (N_R -�� n_r)$ \\ 
\hline
\end{tabular}
\label{t1}
\end{table}

\begin{remark}
The addition of complexes to both sides of a reaction is similar to the technique used by M. Johnston for translating mass action systems to generalized mass action systems in \cite{JOHN2014}.
\end{remark}

\noindent In the next proposition, we provide a proof of a Table \ref{t1} entry which is not straightforward.

\begin{proposition}
\begin{itemize}
\item[i)] $l^{*} = l^{*}_r + l^{*}_b + l$, where $l^{*}_r =$ number of new linkage classes generated by new reactants  and $l^{*}_b =$ number of new linkage classes due to link-breaking. 
\item[ii)] $l^{*} - l^{*}_b \le (N_R - n_r) + l$.   
\end{itemize}
\end{proposition}

\begin{proof} For $i)$, the equation expresses the partitioning into 3 subsets. For $ii)$, a new reactant adds at most 1 linkage class (none if it coincides with an old product complex or at least one of the new product complexes in its linkage class coincides with an old complex).
\end{proof}

The ``link-breaking" effect of CF-RM is shown in the CRN in Figure~\ref{f1}:  if  $R_{(i-1)}: X_1 \rightarrow X_i$ for $i=2,...,5$, $R_5: X_4 \rightarrow X_6$, $R_6: X_5 \rightarrow X_7$ and $X_1$ NF with CF-subsets $\{R_1,R_2\}$ and $\{R_5,R_6\}$, then $\delta^{*} = 10 - 4 - 6 = 0 = \delta$.

\begin{figure}[H]
    \centering
    \includegraphics[width=0.3\textwidth]{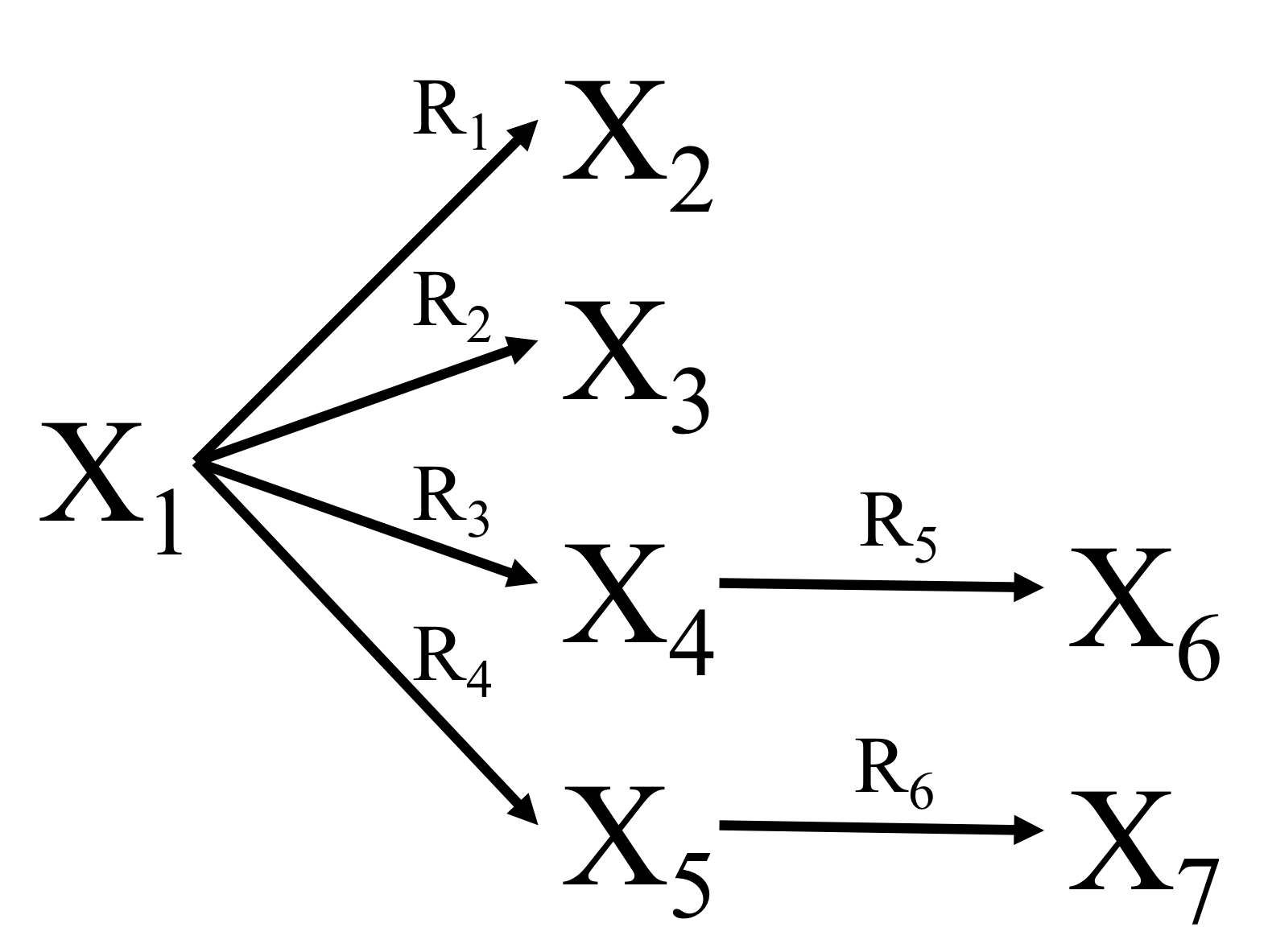}
\caption{The ``link-breaking" effect of CF-RM in the given CRN. }
\label{f1}
\end{figure}

One notes however that three key network numbers of $\mathscr{N}^{*}$  have question marks: the number of complexes $n^{*}$, the deficiency $\delta^{*}$ and the number of  terminal strong linkage classes  $t^*$. Indeed, for many networks, the deficiency increases under CF-RM, but, as the following Proposition shows, for certain network classes, it decreases. 

\begin{proposition}
Let $d$ be an integer $\ge 2$. Let $\mathscr{N}_d$ be the CRN with species $X_1$, $X_2$ and the following reactions:\\
$R_1: X_1 \rightarrow 2X_1$ \\ 
$R_i: X_1 \rightarrow 2iX_1 + X_2$ for $i = 2,...,d$\\
$R_{d+i -1}: X_1 \rightarrow (2i - 1)X_1 + X_2 \rightarrow X_1 + (2i-1)X_2: R_{2d+i -2}$ for $i = 2,...,d$\\
Let $X_1$ be an NF node with CF-subsets $\{R_1,\cdots, R_d\}$ and $\{R_{d+1},\cdots, R_{2d-1}\}$.\\
Then, $\delta - \delta^{*} = d - 1$. 
\end{proposition}
 
\begin{proof} The new reactions are:
$2X_1 \rightarrow 2iX_1 + X_2$ for $i  = 2,..., d$. The remaining reactant complexes are all non-branching, thus RDK and unchanged. Hence there is no new complex, while there are $d -1$ new linkage classes due to the ``link-breaking" effect, i.e., 
$n^{*} = n, l^{*} = 1 + (d -1) = d \Rightarrow \delta^{*} = n - d - 2$. Therefore,  
$\delta - \delta^{*} = (n - 1 - 2) - (n - 2 - d) = d - 1$. 
\end{proof}

In the next section, we present a special variant of CF-RM where these network numbers can be better estimated.

\subsection{CF-RM$_\textbf{+}$: a ``choosier" CF-RM variant}
CF-RM$_+$ is a variant of CF-RM which uses additional criteria in the selection of the new reactant multiples. All other steps are identical with the generic CF-RM method, i.e., a CF-RM$_+$ transform is also a CF- transform.

CF-RM$_+$ chooses the reactant multiple so that 
\begin{itemize}
\item[a)] the new reactant differs from all existing complexes, and
\item[b)] all the new product complexes in the CF-subset also differ from all existing complexes.
\end{itemize}

There are of course various ways of ensuring that conditions a) and b) are fulfilled and we leave it to the first consequence of transforming via CF-RM$_+$, which is a more predictable change in deficiency.

\begin{proposition}
For a CF-RM$_+$ transform $\mathscr{N}^{*}$, $\delta^{*}\ge  \delta$. 
\end{proposition}
\begin{proof}  For any CRN, $n = n_r + t_p$, where $t_p$ is the number of terminal points.  In an CF-RM$_+$ transform, in each subset to be changed, there is one new reactant complex and exactly $x$ new terminal points. The number of reactions to be changed in the CF-subset is also pertained by $x$. Since all terminal point of the original network are conserved (with no coincidence), we obtain $n^{*} - n = (N_R - n_r ) + (r - r_{mcf})$. On the other hand, $l^{*} - l = l^{*}_r + l^{*}_b$. For any CF-RM$_+$ transform, $l^{*}_b \le r - r_{mcf}$ (a link-break is created per new reactant--whether it leads to a new linkage class or not depends on specific network properties). This implies that $l^{*} - l \le (N_R - n_r ) +( r - r_{mcf})$. Hence, $\delta^{*}- \delta = (n^{*}- l^{*}) - (n - l) = (n^{*} - n) - (l^{*} - l) \ge (N_R - n_r) + (r - r_{mcf})- (N_R - n_r ) - (r - r_{mcf} )= 0$. 
\end{proof}

\begin{remark}
The monomolecular system from Figure~\ref{f1} shows this lower bound is sharp.
\end{remark}

Besides the change in deficiency, the change in the number of terminal strong linkage classes is difficult to predict under the generic CF-RM transformation.  Recall that $t$  has two components, i.e., $t = t_p + t_c$, which are the number of terminal points and the number of cycle terminal classes. Under CF-RM$_+$, the relationships for its components can be predicted and together provide an expression for the change in $t$ as shown in the following Proposition:

\begin{proposition}\label{Prop:tptc}
For a CF-RM$_+$ transform $\mathscr{N}^{*}$, we have:\\
\begin{itemize}
\item[i)] $t_p^{*} - t_p = r - r_{mcf}$
\item[ii)] $t_c^{*} - t_c \le  0$
\item[iii)] $t^{*} - t \le r - r_{mcf}$
\end{itemize}
\end{proposition}
\begin{proof} $i)$ was already shown (and used) in the previous Section.
For $ii)$ note that a reversible pair of reactions can be broken up into two irreversible reactions under CF-RM$_+$.  On the other hand, no new cycles can emerge since there is no coincidence of new complexes with existing ones.
$iii)$ follows by adding $i)$ and $ii)$.
\end{proof}

\begin{corollary} 
For a CF-RM$_+$ transform $\mathscr{N}^{*}, n^{*} = n + (N_R - n_r) + (r - r_{mcf})$.
\end{corollary}
\begin{proof} In the identity $n^{*} - n = (n_r^{*} - n_r) + (t_p^{*} - t_p)$, we substitute $N_R$ for $n_r^{*}$ and use Proposition \ref{Prop:tptc}.i.
\end{proof}

\begin{example} (Running Example - Part 2) 
To apply CF-RM to Schmitz's carbon cycle model, we replace $R_3$, $R_4$, and $R_7$ with the following reactions:

$$\begin{array}{lll}
R_3^* &: 2M_1  &\rightarrow M_5 + M_1\\
R_4^* &: 2M_2 &\rightarrow M_1 + M_2\\
R_7^* &: 2M_3 &\rightarrow M_1 + M_3\\
\end{array}$$
Each of the new reactions forms a linkage class of $\mathscr{N}^{*}$, with the remaining original 10 reactions of $\mathscr{N}$ forming the fourth one depicted in Figure~\ref{fig:CFtrans}:

\begin{figure}[H]
    \centering
    \includegraphics[width=0.7\textwidth]{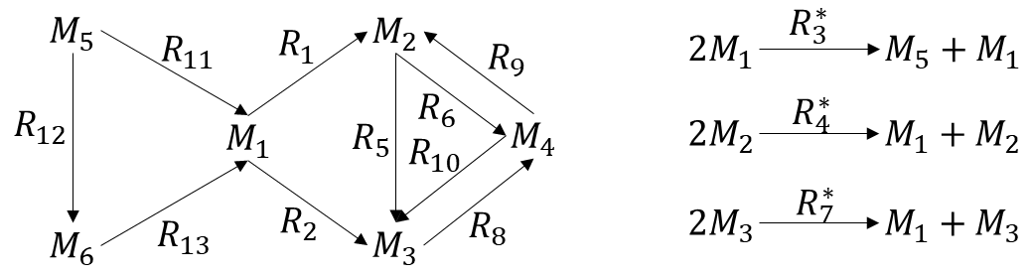}
\caption{The CRN after applying CF-RM to Schimtz's carbon cycle model in Figure~\ref{fig:smpriga}.}
\label{fig:CFtrans}
\end{figure}

\noindent Table~\ref{tableCF} presents the network numbers of the $\mathscr{N}^{*}$.
\begin{table}[H]
\caption{Key network number of CRN $\mathscr{N}^{*}$ of Schimtz's carbon cycle model.}
\center
\begin{tabular}{p{8cm} p{2cm}}
\hline
Network number& Value/bounds \\ 
\hline
Number of species &6  \\ 
\hline
Number of complexes &12 \\ 
\hline
 Number of reactions &13  \\ 
\hline
Number of reactant complexes &9 \\ 
\hline
Number of linkage classes &4 \\ 
\hline
Number of terminal strong linkage classes &4 \\ 
\hline
Deficiency &3  \\ 
\hline
\end{tabular}
\label{tableCF}
\end{table}

\noindent The network is $t$-minimal, but clearly not weakly reversible (in fact, it is point-terminal).  Note that it is also a CF-RM$_+$ transform.\\

 The $T$ matrix of the CF system $(\mathscr{N}^{*}, K^*)$ is given by:\\
 
${T}=$
 \begin{blockarray}{cccccccccc}
&$M_1$ & $M_2$ & $M_3 $& $M_4$ & $M_5$ & $M_6$ & $2M_1$  &$ 2M_2$  & $2M_3$  \\
\begin{block}{c[ccccccccc]}
$M_1$&1 &0 &0 &0 &0 &0 &0.36  &0 &0 \\
$M_2$&0 &1 &0 &0 &0 &0 &0  &9.4 &0 \\
$M_3$&0 &0 &1 &0 &0 &0 &0 &0 &10.2 \\
$M_4$&0 &0 &0 &1 &0 &0 &0 &0 &0 \\
$M_5$&0 &0 &0 &0 &1 &0 &0 &0 &0 \\
$M_6$&0 &0 &0 &0 &0 &1 &0 &0 &0 \\
\end{block}
\end{blockarray}.

\end{example}

%%%%%%%%%%%%%%%%%%%%%%%%%%%%%%
\subsection{A Subspace Coincidence Theorem for NF kinetic systems}

In this Section, we present an initial application of CF-RM transformation by deriving a Subspace Coincidence Theorem for NF systems.

Arceo et al. \cite{csl:1} generalized the Subspace Coincidence Theorem of Feinberg and Horn \cite{csl:3} from MAK systems to CF systems as follows:

\begin{theorem}\label{thm:2}
For a complex factorizable system on a network $\mathscr{N}.$
\begin{itemize}
\item[1)] If $t-l>\delta$, then $K\ne S$.
\item[1')] If $0<t-l=\delta$, and a positive steady state exists, then $K\ne S$. In fact $dim \ S-dim \ K\ge t-l-\delta+1$\\
if the system is also factor span surjective.
\item[2)] If $t-l=0$ (i.e., $\mathscr{N}$ is $t-$minimal), then $K=S$. 
\item[3)] If $0<t-l<\delta$ or $=\delta$ and a positive steady state does not exist, then it is rate constant dependent whether $K=S$ or not.  
\end{itemize}
\end{theorem}
We first note that for any CF-RM transform $\mathscr{N}^{*}$, we not only have coincidence of stoichiometric subspaces $S= S^{*}$ but also coincidence of the kinetic subspaces $K = K^{*}$ (due to the dynamic equivalence, $f = f^*$, implying $\Ima f = \Ima f^{*}$ and $\spn (\Ima f) = \spn (\Ima f^{*})$.

Our approach is to identify properties for an NF system so that its CF-RM$_+$ transform satisfies the conditions of the Theorem above.  Our first step is to extend the kinetics concept of factor span surjectivity, which is currently defined only for CF systems, to any RID kinetic system.

A CF-subset $\mathscr{R}_i$ is characterized by the common interaction map $I_K( \mathscr{R}_ i)$ of the kinetics of its reactions. This leads to the following definition:

\begin{definition}
A RID kinetics is interaction span surjective if and only if the set $\{I_K( \mathscr{R}_ i)\}$ of its CF-subset interaction maps is linearly independent.
\end{definition}

The following Proposition shows that ``interaction span surjectivity" is the correct extension of the factor span surjectivity concept.

\begin{proposition}
If $(\mathscr{N}, K)$ is interaction span surjective, then its CF-transform $(\mathscr{N}^{*}, K^{*})$ is also factor span surjective.
\end{proposition}
 
\begin{proof} Since $(\mathscr{N}^{*}, K^{*})$ is CF, $N_R^{*} = n_r^{*}$. On the other hand, the latter is equal to $N_R$. Hence the set of interaction maps of $\mathscr{N}$ and $\mathscr{N}^{*}$ coincide. For a CF system, since $I_K = \rho' \cdot \psi_K$, it is clear that linear independence of both sets are equivalent.
\end{proof}

As a second step, we identify the network properties of the NF-system  $(\mathscr{N}, K)$ such that the properties needed to apply the various statements of the Theorem to  $(\mathscr{N}^{*}, {K}^{*})$ are ensured.

We first state two Lemmas.

\begin{lemma}
If $\mathscr{N}$   is SRD, then $\mathscr{N}^{*}$ is also SRD.
\end{lemma}
\begin{proof}  $n^{*}_r = N_R \ge n_r \ge s = s^{*}$.
\end{proof}

The second Lemma is a general relationship between TBD and SRD networks derived from a (submitted) manuscript by Farinas et al. entitled ``Species subsets and embedded networks of S-systems":
%\bibitem{csl:2} 
%\newblock  H. Farinas, E. Mendoza and A. Lao,
%\newblock  Species subsets and embedded networks of S-systems, submitted.

\begin{lemma}
Let $\mathscr{N}$ be a chemical reaction network. 
\begin{itemize}
\item[i)] A network with deficiency-bounded terminality (TBD) has sufficient reaction diversity. 
\item[ii)] If the network is point terminal, then the converse also holds, i.e., $TBD\le SRD$ (or equivalently $TND\ge LRD$).
\end{itemize}
\end{lemma}
We can now state and prove a Subspace Coincidence Theorem for NF-systems:

\begin{theorem}\label{thm:3}
Let $(\mathscr{N}, {K})$ be an NF RIDK system.
\begin{itemize}
\item[1)] If $N_R < s$, then $K \ne S$.

\noindent If the system is also intersection span surjective, then either 
\item[2)] $\mathscr{N}$  is $t-$minimal and  $r - r_{mcf}= N_R - n_r$, implies $K = S$; or
\item[3)] $\mathscr{N}$ is TBD and point terminal, implies that $K = S$ is rate-constant dependent.
\end{itemize}
\end{theorem}
\begin{proof}
\begin{itemize}
\item[1)] $n^{*}_r = N_R < s = s'$ means that $\mathscr{N}^{*}$ is LRD. By Lemma 2 $(i)$, it follows that it is also TND, and by (1) of the  KSSC in \cite{csl:21}, $K = K^{*} \ne S^{*} = S$.
\item[2)] In order to apply (2) of the Arceo et al. KSSC \cite{csl:21}, we need to show that there is a CF-RM transform such that $\mathscr{N}^{*}$ is $t-$minimal, or $t^{*} - l^{*} = 0$.   We calculate this difference for an CF-RM$_+$ transform as follows:  $t^{*} - l^{*} = t^{*}_p + t^{*}_c - l^{*} = r  - r_{mcf} + t_p + t^{*}_c- l^{* }$(by Proposition \ref{Prop:tptc}) $=  r  - r_{mcf} + t_p + t^{*}_c - ((N_R-n_r) + t_p + t_c)$ (based on the properties of CF). After canceling terms, we obtain    $0 \le  t^{*} - l^{*} = t^{*}_c - t_c  \le 0$ (by Proposition \ref{Prop:tptc}), implying the claim.
\item[3)] $\mathscr{N}$ is point terminal $\Rightarrow \mathscr{N}^{*}$ is point terminal ( by Proposition \ref{Prop:tptc}.ii), hence after Lemma 1 and Lemma 2 $(ii)$, $\mathscr{N}^{*}$ is also  TBD, and (3) of Arceo et al. Theorem can be applied, implying $K = K^{*} = S^{*} = S$ is rate constant dependent. 
\end{itemize}
\end{proof}

\begin{remark}
Since both the stoichiometric and reactant subspaces of an NF system and its CF-RM transform coincide, the underlying networks have the same R and S class introduced in \cite{csl:21}. This implies that a Theorem for the coincidence of kinetic and reactant subspaces of NF systems analogous to that for CF-systems derived in \cite{csl:21} can also be stated and proved.
\end{remark}

\section{Linear conjugacy of RID kinetic systems}
In this Section, we present a solution to the problem of finding linear conjugates of any RID kinetic system. After extending the Johnston-Siegel Criterion (JSC) for linear conjugacy to CF  systems, we can generate linear conjugates for any RID kinetic system by applying the JSC to any CF-RM transform of the given system. We also discuss some computational challenges regarding the solution approach.
\subsection{The Johnston-Siegel Criterion for linear conjugacy (JSC) of CF kinetic systems}
\begin{theorem}\label{thm:4}
Consider two CF systems $\left( \mathscr{N}, {K}\right)$ and $\left( \mathscr{N'}, {K'}\right)$ with $\mathscr{N} = \left( \mathscr{S}, \mathscr{C}, \mathscr{R} \right)$ and $\mathscr{N}' = \left( \mathscr{S'}, \mathscr{C'}, \mathscr{R'} \right)$. Let $Y = Y'$ be the matrix of complexes for both networks. Suppose further that the factor maps coincide, i.e., $\psi_K = \psi_{K}'$.  Let $A_b$ be a Laplacian with the same structure as that of $\left(\mathscr{N'}, \mathscr{K'}\right)$ and $c$,  a positive vector in $\mathbb{R}^m$ such that   $Y\cdot A_k = C\cdot Y\cdot A_b$ , where $C = \diag (c)$.
Then $\mathscr{N}$   is linearly conjugate to $\mathscr{N'}$ with the Laplacian  $A_k' = A_b \cdot \diag (\psi_{K'(c)})$.
\end{theorem}

\begin{proof} Let $\varphi \left( {{x_o},t} \right)$ be the solution of the system of ODE $\dot{x}  = f\left( x \right) = Y\cdot{A_k}\cdot{\psi _K}$ associated to the reaction network $\mathscr N$.\\ Consider the linear map $h\left( x \right) = {C^{ - 1}}x$  where $C = diag(c)$, $c \in \mathbb{R}_{ > 0}^n$. \\ Let $\tilde\varphi  \left( {{y_0},t} \right) = {C^{ - 1}}\varphi \left( {{x_0},t} \right)$ so that $\varphi \left( {{x_0},t} \right) = C\tilde\varphi  \left( {{y_0},t} \right)$. It follows that
\begin{align}
  \tilde \varphi '\left( {{y_0},t} \right) &= {C^{ - 1}}\cdot\varphi '\left( {{x_0},t} \right) \hfill \nonumber \\
 & = {C^{ - 1}}\cdot Y\cdot{A_k}\cdot{\psi _K}\left( {\varphi \left( {{x_0},t} \right)} \right) \hfill \nonumber\\
  &= {C^{ - 1}}\cdot C\cdot Y\cdot{A_b}\cdot{\psi _K}\left( {C\tilde\varphi  \left( {{y_0},t} \right)} \right) \hfill  \nonumber 
\end{align}
Now, 
\begin{align}
  \psi_K \left( {C\tilde \varphi \left( {{y_0},t} \right)} \right) &= {\psi _K}\left( {diag(c)\tilde \varphi  \left( {{y_0},t} \right)} \right) \hfill \nonumber \\
  &= D\cdot{\psi _K}\left( {\tilde\varphi  \left( {{y_0},t} \right)} \right) \nonumber \hfill 
\end{align}
where $D = diag (e)$ and ${e_j} = \left\{ {\begin{array}{*{20}{c}}
  {{c^{{F._j}}}{\text{ , if complex }}j{\text{ is a reactant of some reaction }}k{\text{ }}} \\ 
  {1{\text{, otherwise}}} 
\end{array}} \right.$\\
So, $\tilde \varphi ' \left( {{y_0},t} \right) = Y\cdot{A_b}\cdot D\cdot{\psi _K}\left( {\tilde \varphi \left( {{y_0},t} \right)} \right)$. Clearly, $\tilde \varphi \left( {{y_0},t} \right)$ is a solution of the system $\dot{x}  = Y\cdot{A_b}\cdot D\cdot{\psi _K}$ corresponding to the reaction network $\mathscr N'$. We have that $h\left( {\varphi \left( {{x_0},t} \right)} \right) = \tilde \varphi \left( {h\left( {{x_0}} \right),t} \right)$ for all ${x_0} \in \mathbb{R}_{ > 0}^n$ and $t\ge 0$ where $y_0=h(x_0)$ since ${y_0} = \tilde \varphi \left( {{y_0},t} \right) = {C^{ - 1}}{y_0} = \varphi \left( {{y_0},t} \right)$. It follows that networks $\mathscr N$ and $\mathscr N'$ are linearly conjugate.
\end{proof}

\subsection{A solution to the linear conjugacy problem of RID kinetic systems}
A solution approach to the linear conjugacy problem of RID kinetic systems is clearly to first transform the system if necessary (i.e., if it is an NF system) via CF-RM to a CF system and then apply the Johnston-Siegel Criterion to generate linearly conjugate systems. The second step could be done using MILP algorithms based on the JSC, once these are extended to appropriate CF systems (cf. Section 6).

\begin{example} (Running Example - Part 3)   
In \cite{FMRL2019}, Fortun et al.~derived a Deficiency Zero Theorem for a class of NF power law kinetic systems and applied it to a subsystem of the Schmitz's carbon cycle model to establish the existence of positive equilibria for the subsystem. The authors then used a ``Lifting Theorem" of \cite{JOSH2013} to show the existence of corresponding positive equilibria for the whole system. Here, we provide an alternative approach for this result by using the MILP algorithm of \cite{csl:8}, a special case of the MILP algorithm introduced in Section~6, to construct a weakly reversible PL-TIK system, which is linearly conjugate to the CF-transform of Schmitz's model discussed previously. The results of \cite{csl:7} show that this weakly reversible system has positive equilibria, and hence so does its linear conjugate, the Schmitz's carbon cycle model.
\end{example}

\noindent The sparse linear conjugate of $(\mathscr{N}^*, K^*)$ was obtained using the MILP algorithm, described in \cite{csl:8}. The algorithm seeks to generate linearly conjugate realizations for a class of power-law kinetic systems, i.e., PL-RDK. Prior to the implementation of the algorithm, the map of complexes $Y$, the Laplacian map $A_k$, and kinetic order matrix $F$ are required to be set first. The matrix $F$ was given in the preceding section. The following are the  associated matrices $Y$ and $A_k$ of the system.

$Y=$
 \begin{blockarray}{ccccccccccccc}
&$C_1$ &$ C_2$  & $C_3$ & $C_4$ & $C_5 $& $C_6$ & $C_7$ & $C_8$ & $C_9$ & $C_{10}$ & $C_{11}$ & $C_{12}$\\
\begin{block}{c[cccccccccccc]}
$M_1$ &1	 &0	 &0	 &2	 &1	 &0	 &1	 &0	 &0	 &1	&0	&0\\
$M_2$ &0	 &1	 &0	 &0	 &0	 &2	 &1	 &0	 &0	 &0	&0	&0\\
$M_3$ &0	 &0	 &1	 &0	 &0	 &0	 &0	 &0	 &2	 &1	&0	&0\\
$M_4$ &0	 &0	 &0	 &0	 &0	 &0	 &0	 &1	 &0	 &0	&0	&0\\
$M_5$ &0	 &0	 &0	 &0	 &1	 &0	 &0	 &0	 &0	 &0	&1	&0\\
$M_6$ &0	 &0	 &0	 &0	 &0	 &0	 &0	 &0	 &0	 &0	&0	&1\\
\end{block}
\end{blockarray}

$A_k=$
 \begin{blockarray}{ccccccccccccc}
&$C_1$ &$ C_2$  & $C_3$ & $C_4$ & $C_5 $& $C_6$ & $C_7$ & $C_8$ & $C_9$ & $C_{10}$ & $C_{11}$ & $C_{12}$\\
\begin{block}{c[cccccccccccc]}
$C_1$ &-0.12	&0	&0	&0	&0	&0	&0	&0	&0	&0	&0.086	&0.03\\
$ C_2$  &0.09	&-0.10	&0	&0	&0	&0	&0	&0.002	&0	&0	&0	&0\\
$C_3$ &0.03	&0.08	&-0.71	&0	&0	&0	&0	&0.001	&0	&0	&0	&0\\
$C_4$ &0	&0	&0	&-10.09	&0	&0	&0	&0	&0	&0	&0	&0\\
$C_5$& 0	&0	&0	&10.09	&0	&0	&0	&0	&0	&0	&0	&0\\
$C_6$ &0	&0	&0	&0	&0	&-0.70	&0	&0	&0	&0	&0	&0\\
 $C_7$ &0	&0	&0	&0	&0	&0.70	&0	&0	&0	&0	&0	&0\\
$C_8$ &0	&0.016	&0.71	&0	&0	&0	&0	&-0.003	&0	&0	&0	&0\\
$C_9$ &0	&0	&0	&0	&0	&0	&0	&0	&-0.2	&0	&0	&0\\
$C_{10}$ &0	&0	&0	&0	&0	&0	&0	&0	&0.2	&0	&0	&0\\
 $C_{11}$ &0	&0	&0	&0	&0	&0	&0	&0	&0	&0	&-0.17	&0\\
$C_{12}$ &0	&0	&0	&0	&0	&0	&0	&0	&0	&0	&0.09	&-0.03\\
\end{block}
\end{blockarray}

\noindent where $C_1: M_1$, $ C_2: M_2$, $C_3: M_3$, $C_4: 2M_1$, $C_5: M_5+M_1 $, $C_6: 2M_2$, $C_7: M_1+M_2$, $C_8: M_4$, $C_9: 2M_3$,  $C_{10}: M_1+M_3$, $C_{11}: M_5$, and $C_{12}: M_6$.\\

Additionally, the parameters were set as follows: $\epsilon=0.001$ and $\displaystyle{u_{ij}=20, i,j=1,2,...,12, i\ne j}$. Using MATLAB R2018b, the linearly conjugate weakly reversible sparse realization ($\tilde{\mathscr{N}},\tilde{ K}$)  was obtained with the corresponding Laplacian map $A_k^{sparse}$. 

$A_k^{sparse}=$
 \begin{blockarray}{ccccccccccccc}
&$C_1$ &$ C_2$  & $C_3$ & $C_4$ & $C_5 $& $C_6$ & $C_7$ & $C_8$ & $C_9$ & $C_{10}$ & $C_{11}$ & $C_{12}$\\
\begin{block}{c[cccccccccccc]}
$C_1$&-0.12&0	&0	&0	&0	&1.05	&0	&0	&0.33	&0	&0	&0\\
$ C_2$&0	&-0.10	&0	&0	&0	&0	&0	&0.002	&0	&0	&0	&0\\
$ C_3$&0	&0.08	&-0.71	&0	&0	&0	&0	&0.001	&0	&0	&0	&0\\
$ C_4$&0	&0	&0	&-2.98 	&0	&0	&0	&0	&0	&0	&0.09	&0.03\\
$ C_5$&0	&0	&0	&0	&0	&0	&0	&0	&0	&0	&0	&0\\
$ C_6$&0.09	&0	&0	&0	&0	&-1.05	&0	&0	&0	&0	&0	&0\\
$ C_7$&0	&0	&0	&0	&0	&0	&0	&0	&0	&0	&0	&0\\
$ C_8$&0	&0.02	&0.71	&0	&0	&0	&0	&-0.003	&0	&0	&0	&0\\
$ C_9$&0.03	&0	&0	&0	&0	&0	&0	&0	&-0.33	&0	&0	&0\\
$ C_{10}$&0	&0	&0	&0	&0	&0	&0	&0	&0	&0	&0	&0\\
$ C_{11}$&0	&0	&0	&2.98	&0	&0	&0	&0	&0	&0	&-0.17	&0\\
$ C_{12}$&0	&0	&0	&0	&0	&0	&0	&0	&0	&0	&0.09	&-0.03\\
\end{block}
\end{blockarray}

The linear conjugacy constants are $c_1=2.28$, $c_2=1.14$, $c_3=1.14$, $c_4=1.14$, $c_5=4.56$, and $c_6=4.56$. Furthermore, the associated system of ODEs is given below: 
 \begin{equation*}
\begin{split}
\frac{dM_1}{dt} &=  -0.1242M_1-5.953M_1^{0.3572}+1.052M_2^{9.4}+0.334M_3^{10.2}+0.172M_5+0.067M_6 \\ 
\frac{dM_2}{dt} &= 0.186M_1-0.095M_2-2.104M_2^{9.4}+.002M_4 \\ 
\frac{dM_3}{dt} &=0.062M_1+0.078M_2-2*0.334M_3^{10.2}-0.714M_3+0.001M_4\\ 
\frac{dM_4}{dt} &=0.016M_2+0.714M_3-0.003M_4\\ 
\frac{dM_5}{dt} &= 2.977M_1^{0.3572}-0.172M_5\\ 
\frac{dM_6}{dt} &=  0.0862M_5-0.0333M_6 \\ 
\end{split}
\end{equation*}

The network $\tilde{\mathscr{N}}$ and its numbers are shown in the following Figure~\ref{sparseschmitz} and Table~\ref{table:sparseschmitz}:

\begin{figure}[h!]
    \centering
    \includegraphics[width=0.75\textwidth]{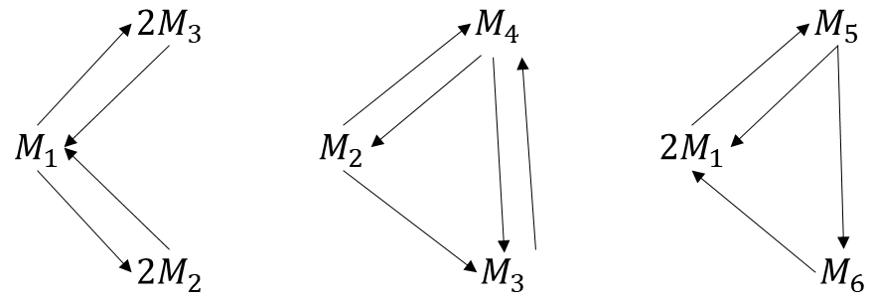}
\caption{Sparse linearly conjugate realization.}
\label{sparseschmitz}
\end{figure}

\begin{table}[H]
\caption{Network number of the sparse network  $\tilde{\mathscr{N}}$ . }
\center
\begin{tabular}{p{8cm} p{2cm}}
\hline
Network number& Value/bounds \\ 
\hline
Number of species & 6\\
\hline
Number of complexes &9\\
\hline
Number of reactions & 13\\
\hline
Number of reactant complexes &9\\
\hline
Number of linkage classes  & 3\\
\hline
Number of terminal strong linkage classes  &3\\%4 or 3?
\hline
Rank  &5\\ 
\hline
Deficiency  &1\\ 
\hline
\end{tabular}
\label{table:sparseschmitz}
\end{table}

The $\hat{T}$ matrix of the system is given by:\\

$\hat{T}=$
 \begin{blockarray}{cccccccccc}
&$M_1$ &$ 2M_2$  & $2M_3$ & $M_2$ & $M_3 $& $M_4$ & $2M_1$ & $M_5$ & $M_6$ \\
\begin{block}{c[ccccccccc]}
$M_1$&1 &0 &0 &0 &0 &0 &0.36 &0 &0 \\
$M_2$&0 &9.4 &0 &1 &0 &0 &0 &0 &0\\
$M_3$&0 &0 &10.2 &0 &1 &0 &0 &0 &0\\
$M_4$&0 &0 &0 &0 &0 &1 &0 &0 &0 \\
$M_5$&0 &0 &0 &0 &0 &0 &0 &1 &0 \\
$M_6$&0 &0 &0 &0 &0 &0 &0 &0 &1\\
$L_1$&1 &1 &1 &0 &0 &0 &0 &0 &0 \\
$L_2$&0 &0 &0 &1 &1 &1 &0 &0 &0 \\
$L_3$&0 &0 &0 &0 &0 &0 &1 &1 &1 \\
\end{block}
\end{blockarray}

One readily computes that it has maximal rank, 9, and hence $(\tilde{\mathscr{N}}, \tilde{K})$ is a PL-TIK system.  Since each of the linkage classes has zero deficiency, according to the Deficiency Zero Theorem for PL-TIK systems (Theorem 5 and Corollary 6 , \cite{csl:7}), each subsystem possesses positive equilibria. It then follows from Theorem 4 of \cite{csl:7} that the whole system also has positive equilibria. Hence, the linearly conjugate system $(\mathscr{N}, K)$  also has positive equilibria, which are necessarily complex balanced since the system has zero deficiency. The graphs of the individual trajectories of  ($\mathscr{N}^ *, K^*$)  and $(\tilde{\mathscr{N}}, \tilde{K})$ are depicted in Figure \ref{f3}.

\begin{figure}
\center
\subfigure[]{
\resizebox*{4cm}{!}{\includegraphics{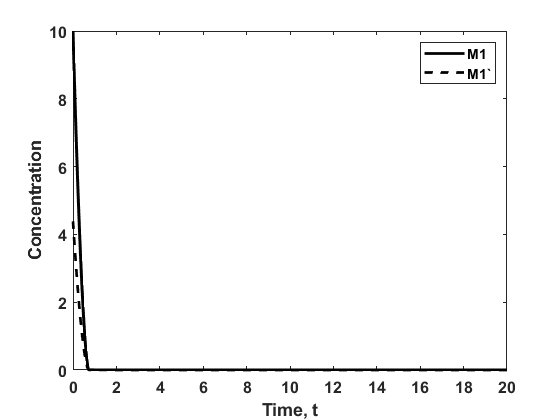}}}\hspace{5pt}
\subfigure[]{
\resizebox*{4cm}{!}{\includegraphics{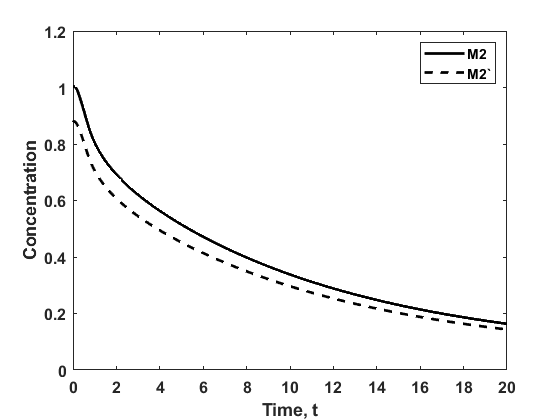}}}\hspace{5pt}
\subfigure[]{
\resizebox*{4cm}{!}{\includegraphics{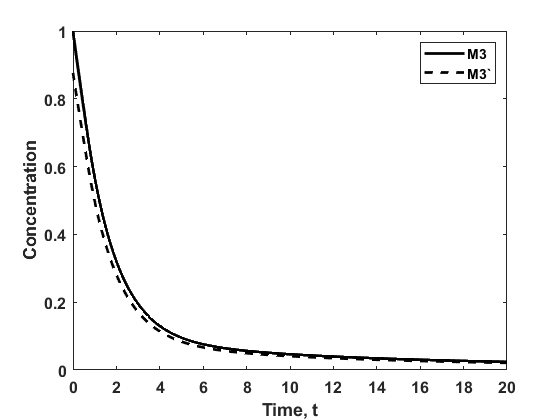}}}\hspace{5pt}
\subfigure[]{
\resizebox*{4cm}{!}{\includegraphics{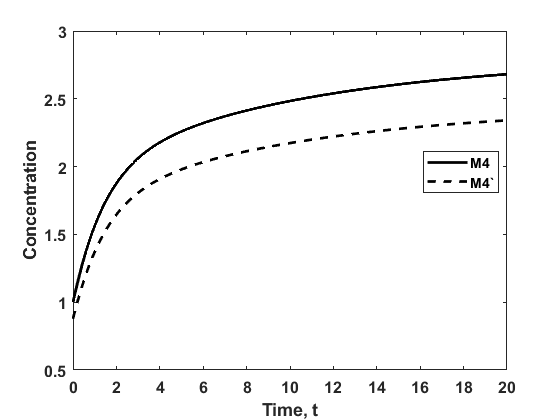}}}\hspace{5pt}
\subfigure[]{
\resizebox*{4cm}{!}{\includegraphics{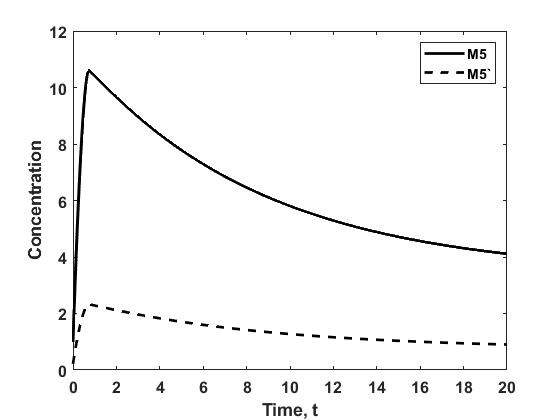}}}\hspace{5pt}
\subfigure[]{
\resizebox*{4cm}{!}{\includegraphics{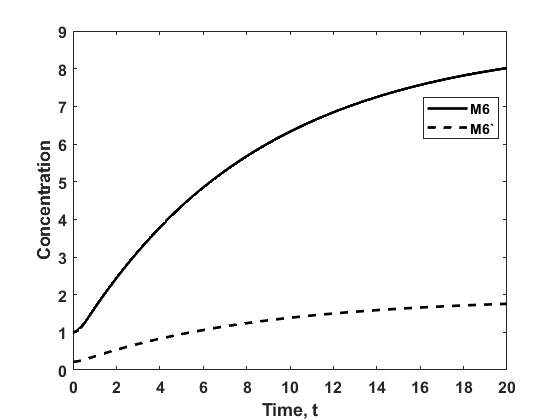}}}
\caption{The graphs of the trajectories for  ($\mathscr{N}^ *, K^*$)  and $(\tilde{\mathscr{N}}, \tilde{K})$. $M_i'$ represents a trajectory in the sparse realization.}
\label{f3}
\end{figure}

There are however several challenges with this ``solution in principle": It may be difficult to compute the CF subsets of a RID kinetic system, which form the basis of the  CF-RM method, as it involves determining if interaction functions (for an infinite number of domain values) are equal. This clarity depends on how explicit and complex the functional expressions are. Similarly, applying the JSC to a CF system, one needs to establish the equality of the factor maps, which is equivalent to the difficulty with interaction functions cited above.\\

In the next section, we identify a large subset of RID kinetics, where the solution approach can be applied in general.

\section {Linear conjugacy of RIP kinetic systems}
This section introduces the large subset of RID kinetics with interaction parameter maps (RIPK). The subset includes power law kinetics (PLK), Hill-type kinetics (HTK)--originally called ``Saturation-Cooperativity" (SC) Formalism \cite{csl:6}, and other published biochemical kinetics such as linlog  \cite{csl:5} and loglin kinetics  \cite{csl:4}. We extend the $T$ matrix concept of \cite{csl:7} to complex factorizable RIP kinetics (denoted by RIP-CFK) and obtain a computationally feasible form of the JSC for this kinetics set, which leads to executable solutions of the linear conjugacy problem. 
\subsection{RIP kinetics: RID kinetics with interaction parameter map}
\begin{definition} A set $\mathscr{K} \in$  RIDK is said to be of type ``RID kinetics with interaction parameter maps" if there is a family of maps $ \left\{ p_{\mathscr{K}}: \mathscr{R} \rightarrow R^{m1} \times ... \times R^{mk} \mid K \in \mathscr{K}\right\}$ such that
\begin{itemize}
\item[i)] $p_K(r) =  p_K(r')  \Rightarrow  I_K(x)_r = I_K(x)_{r'}$ for all $x$ in $\Omega$ and
\item[ii)]$p_K = p_{K'} \Rightarrow I_K(x) = I_{K'}(x)$ for all $x$ in $\Omega$
\end{itemize}
\end{definition}

\begin{example}
PLK with the family of kinetic order matrices, i.e., $p_\mathscr{K}(r)= F_r$, (kinetic order row vector or interaction), is the primary example. Since $I_K(x) = x^F$, the properties i) and ii) are straightforward.
\end{example}

\begin{example}
Hill-type kinetics (HTK)--originally called ``Saturation-Cooperativity (SC) Formalism" in 2007 by Sorribas et al. \cite{csl:6}. We recall the definition of \cite{csl:1}:
\end{example}

\begin{definition} \textbf{Hill-type kinetics (HTK)} is defined as follows: \\
$$K_j(c)=k_j{\displaystyle \prod_{i=1}^{n}\frac{c^{v_{j,i}}_i}{d_{j,i}+c^{v_{j,i}}_i}}$$\\

\noindent with $c\in \mathbb{R}^n_\ge $ (defined by continuity at the boundary), $k_j\in \mathbb{R}_>$, $d_j\in \mathbb{R}^n_\ge $ and $v_j\in \mathbb{R}^n$ for $j=1,...,m$. Note that the $v_j$ have to be nonegative.
\end{definition}

The family of interaction parameter maps is given by $P_K: \mathscr{R} \rightarrow \mathbb{R}^m \times \mathbb{R}^m$  with $p_K(r) = (v_1,...,v_m,d_1,...,d_m)$, where we leave out the index $j$.

\subsection{CF-RM for RIP-NFK and the JSC for RIP-CFK} 
Since under CF-RM, there is a bijection $\eta : \mathscr{R} \rightarrow \mathscr{R}^*$, if $(\mathscr{N}, {K})$ is an NF RIP kinetic system, then $(\mathscr{N}^*, {K})$ is a CF RIP kinetic system with the interaction parameter map $p_{K}^* (\eta (r)) := p_K(r)$.\\

We denote the set of all complex factorizable kinetics with interaction parameter maps with \textbf{RIP-CFK}.\\

For an interaction parameter map $p_K: \mathscr{R} \rightarrow R_{m1} \times ... \times R_{mk}$, we write $p = m1 + ...+ mk$. It is now easy to formally introduce the $T$ matrix of a RIP-CFK kinetics:

\begin{definition}
The T matrix of a RIP-CFK kinetics $K$ is the $p \times n_r$ matrix whose jth column is $p_K(r)^T$, where $\rho(r) = j$. The $\hat{T}$ matrix is the $(p + l) x$  given by adjoining the characteristic functions of the linkage classes as rows to the T matrix.  The rank of the $\hat{T}$ matrix is denoted by $\hat{q}$.
\end{definition}

We have the following useful Proposition:

\begin{proposition}
Let $(\mathscr{N}, {K})$ and $(\mathscr{N}', K')$ be RIP-CFK systems. If $T = T'$, then $\psi_K =  \psi_{K'}$.
\end{proposition}

\begin{proof} $T = T' \Rightarrow p_K = p_{K'}$ for all $K, K'$ of the same type $\Rightarrow I_K = I_{K'}$  (by definition of interaction parameter map) $\Leftrightarrow \psi_K =  \psi_{K'}$ (since the maps differ only with the reactions map).
Hence, RIP-CF kinetics, it suffices to check a finite set of vectors to establish the coincidence of the factor maps. This allows the extension of the JSC-based MILP algorithms for PL-RDK systems to RIP-CFK systems.
Since the CF-RM transform of a RIP-NFK system is clearly a RIP-CFK system, we obtain a general  computational solution for the linear conjugacy of RIP kinetic systems. 
\end{proof}
\begin{remark}
The set $\{\mathscr{K} \in RIPK \mid \rho(r) = \rho(r') \Rightarrow p_K(r) = p_K(r')\}$ may, in general, be a proper subset of RIP-CFK. This may result in computing a smaller set of linear conjugates as when the whole set RIP-CFK is used. This is a small price one pays for ensuring the computational feasibility. There are, however, various RIP kinetics for which the converse $\psi_K(x) = \psi_{K'}(x) \Rightarrow p_K(r) = p_K(r')$ also holds, so that the corresponding sets are equal. Examples are PLK and $PYK_h$ (the set of poly-PL kinetics with h summands), which form a covering of PYK (cf. a manuscript in preparation by Talabis et al. entitled ``A Weak Reversibility Theorem for poly-PL kinetics and the replicator equation").\\

In  \cite{csl:1}, we introduced the notations PL-RDK and HT-RDKD for the subsets of PLK and HTK respectively, which satisfy $\rho(r) = \rho(r') \Rightarrow p_K(r) = p_K(r')$.  For any other subset A  of RIPK, we will denote $ \{\mathscr{K} \in RIPK \mid \rho(r) = \rho(r') \Rightarrow p_K(r) = p_K(r')\}$ with A-RDP (kinetics with reactant-determined parameter maps). This notation is consistent with earlier ones since the corresponding letters there indicate the specific parameter maps, too.
\end{remark}

\section{Extension of MILP algorithms to RIP-CFK systems}
Cortez et al. \cite{csl:8} extended the MILP algorithm developed by Johnston et al. \cite{csl:16} to find linearly conjugate networks of PL-RDK systems. Aside from linear conjugacy, other desirable properties can be incorporated in the algorithm such as weak reversibility and minimal deficiency (e.g. deficiency zero). In this study, we focus on extending the algorithm to find linearly conjugates of RIP-CFK kinetic systems. 

\subsection{Key components of the MILP algorithm }

The algorithm considers two CF systems: the original system $(\mathscr{N}, {K})$ with $\mathscr{N} = (\mathscr{S}, \mathscr{C}, \mathscr{R})$ and
the target system $(\mathscr{N}', {K}')$ with $\mathscr{N}' = (\mathscr{S}', \mathscr{C}', \mathscr{R}')$. The algorithm determines the corresponding network structure of the target system satisfying the linear conjugacy property.  The two networks $\mathscr{N}$  and $\mathscr{N}'$ have the same set of species and complexes. As a consequence, their corresponding molecularity matrices and the coefficient maps coincide. The algorithm requires that $\mathscr{R}$ and ${K}$ be known while $\mathscr{R}'$ and ${K}'$ are to be obtained. The following are needed to be ascertained prior to the MILP implementation: 
\begin{itemize}
\item molecularity matrix $Y\in \mathbb{R}^{m\times n}_{\ge 0}$;
\item matrix $M=Y \cdot A_k$, where $A_k$ is the Laplacian map;
\item parameter $\epsilon >0$, that is set to be sufficiently small; and  
\item parameter $u_{ij}>0$, where $i,j=1,...,m$, $i \ne j$.
\end{itemize}

\begin{remark} 
Note that $\epsilon$ and $u$ are introduced to ensure the correct structure of the linearly conjugate realization. 
\end{remark}

\subsection{MILP algorithm to CF systems}

The MILP algorithm finds a sparse linearly conjugate realization of the original network $\mathscr{N}$ . A sparse realization contains the minimum number of reactions, hence the associated objective function of the MILP model is \\ 
\begin{align}
\text{Minimize} \displaystyle \sum^m _{i,j=1}\delta_{ij}.
\end{align}
There are two sets of constraints in the model which indicate the linear conjugacy condition and desired structure of the network.  \\ 
\begin{align}
\left( {\bf \text LC} \right)&\left\{ \begin{gathered}
  Y \cdot {A_b} = {C^{ - 1}} \cdot M,{\text{ }}C = diag\left\{ c \right\} \hfill \\
  \sum\limits_{i = 1,i \ne j}^m {\left[ {{A_b}} \right]} {}_{ij} = 0,j = 1,...,m \hfill \\
  {\left[ {{A_b}} \right]_{ij}} \geqslant 0,{\text{for }}i = 1,...,m,{\text{  }}i \ne j \hfill \\
  {\left[ {{A_b}} \right]_{ii}} < 0,{\text{ for }}i = 1,...,m \hfill \\
   \epsilon  \leqslant {c_i} \leqslant \frac{1}{\epsilon},{\text{  for }}i = 1,...,n \hfill \\ 
\end{gathered}  \right.\label{equation:constraint1}
\end{align}
\begin{align}
\left( {\bf \text LC - S} \right)&\left\{ \begin{gathered}
  0 \leqslant  - {\left[ {{A_b}} \right]_{ij}} + {u_{ij}} \cdot {\delta _{ij}},i,j = 1,...,m,{\text{  }}i \ne j \hfill \\
  0 \leqslant {\left[ {{A_b}} \right]_{ij}} -  \epsilon  \cdot {\delta _{ij}},i,j = 1,...,m,{\text{  }}i \ne j \hfill \\
  {\delta _{ij}} \in \left\{ {0,1} \right\}{\text{  for }}i,j = 1,...,m,{\text{  }}i \ne j \hfill \\ 
\end{gathered}  \right.\label{equation:constraint2}
\end{align}

\begin{table*}[h!]
\caption{List of variables used in the MILP}
\label{table3} 
\begin{tabular}{ll}
\hline\noalign{\smallskip}
Notation & Description  \\
\noalign{\smallskip}\hline\noalign{\smallskip}
${\delta _{ij}},i,j = 1,...,m$ & binary variable that keeps track of the presence of the reaction in the target network \\
${\left[ {{A_b}} \right]_{ij}},i,j = 1,2,...,m$ & kinetic matrix with the same structure as the target network\\
$c$ & a vector which is an element of $\mathbb{R}_{ > 0}^n$ \\
$C$ & a diagonal matrix $diag(c)$ with vector $c \in \mathbb{R}_{ > 0}^n$\\
\noalign{\smallskip}\hline
\end{tabular}
\end{table*}
Table \ref{table3} shows the description of the variables used in the model. Constraint (\ref{equation:constraint1}) imposes the linear conjugacy specification while constraint (\ref{equation:constraint2}) ensures that the target network $\mathscr{N}'$ has the correct structure. A dense linearly conjugate network can also be determined by considering the maximization problem analog. \\

The optimal solution (if it exists) of the MILP would yield the matrix $A_b$ with the same structure as $\mathscr{N}'$, and the conjugacy constant vector $c$. The Laplacian map $A_k'$ of the target network is computed as:  
\begin{center}
$A_k'=A_b\cdot D$
\end{center}
where $D=diag(e)$  and ${e_j} = \left\{ {\begin{array}{*{20}{c}}
{{c^{F \cdot j}},{\text{if \ complex $j$ is  a reactant of some reaction }}k}\\
{1,\text{otherwise}}
\end{array}} \right.$

\subsection{MILP algorithm to RIP-CFK systems}

It is important to note that the algorithm developed by Cortez et al. \cite{csl:8} is only applicable to  CF systems (e.g. PL-RDK). For NF systems, the MILP cannot be immediately utilized to generate linearly conjugate realizations. It is necessary to transform it into a CF system through the CF-RM algorithm described in Section 5. This framework is applicable to RIPK systems which include both the power law kinetics (PLK) and Hill-type kinetics (HTK). The computation of the matrix $A_b$ and linear conjugacy vector $c$ is the same for both systems. The process of finding linearly conjugate realizations differs only in the derivation of corresponding sytem of ODEs wherein the respective kinetic order matrix/interaction parameter matrix is incorporated accordingly. Additionally, to obtain a proper form of the rational terms in the target HTK system, a linear scaling of the variable of rational term must be carried out, that is the variable must be multiplied by its corresponding linear conjugacy constant. This approach is similar to the approach of \cite{csl:19a} to linear conjugacy of bio-CRNs.

\section{Application to Hill-type kinetic system}
In  \cite{csl:1} and \cite{csl:9}, we introduced CRN representations of GMA systems--as defined in Biochemical Systems Theory (BST)--by means of the biochemical maps usually used to define them. These representations are actually independent of the power law kinetics assigned to the reactions from BST and we will use them for other RID kinetic systems too, as illustrated in the following examples. \\

In the following, after a brief review of Hill-type kinetics, we consider a reference metabolic system of \cite{csl:6}. We apply the MILP algorithm to Hill-type kinetics and compare the set of linear conjugates with those of power law kinetics on the same chemical reaction network.

\subsection{Review of Hill-type kinetics}
The set of Hill-type kinetics was introduced in 2007 by Sorribas et al. \cite{csl:6} under the name of ``Saturation-Cooperativity Formalism" (SC-Formalism). This framework generalizes the well-known Michaelis-Menten and Hill functions in one variable. The term ``Hill-type kinetics" (HTK) was introduced in 2013 in the paper of Wiuf and Feliu \cite{csl:20}. In  \cite{csl:1}, it was shown that a Hill-type kinetics can be written as follows: \\
Given 
\begin{itemize}
\item $e_j:\Omega_K\rightarrow \mathbb{R}^\mathscr{S}_>$ with $e_j(x)=(x^{\rho(j)_1}_1,..., x^{\rho(j)_m}_m), r_j\in\mathscr{R}$
\item $d_j: \mathbb{R}^\mathscr{S}_>\rightarrow \mathbb{R}^\mathscr{S}_>$ with $d_j(x)=x+D_j$, $r_j\in \mathscr{R}$
\item $m: \mathbb{R}^\mathscr{S}_>\rightarrow \mathbb{R}_>$ with $m(x)=\Pi x_i, i=1,...,m.$ 
\end{itemize}
then $I_H = I_1/I_2$ , with $I_1$ a PLK interaction map with kinetic order matrix $F$ and $(I_2 (x ))_ j = m \cdot d_j \cdot e_j (x)$.
Furthermore,  the dissociation vectors $d_j$ (s. Definition 22) were organized in an $r \times m$ matrix called the ``��dissociation matrix" and the set of complex factorizable Hill-type kinetics was denoted by HT-RDKD (Hill-type with reactant-determined kinetic and dissociation), expressing the fact that it is the pre-image of the interaction parameter map given by the kinetic order and dissociation matrices.\\

\begin{remark}
The method for determining linear conjugates for Bio-CRNs in \cite{csl:18} is applicable to HTK if the exponents are non-negative integers.
\end{remark}

\subsection{The reference system with Hill-type kinetics}
Now, we apply the integrated algorithm to a particular biological system. Specifically, we consider a metabolic network with one positive feedforward and a negative feedback (see Figure \ref{f2}) taken from the published work of  \cite{csl:6}.\\ 

\begin{figure}[!htb]
\centering
    \includegraphics[width=0.45\textwidth]{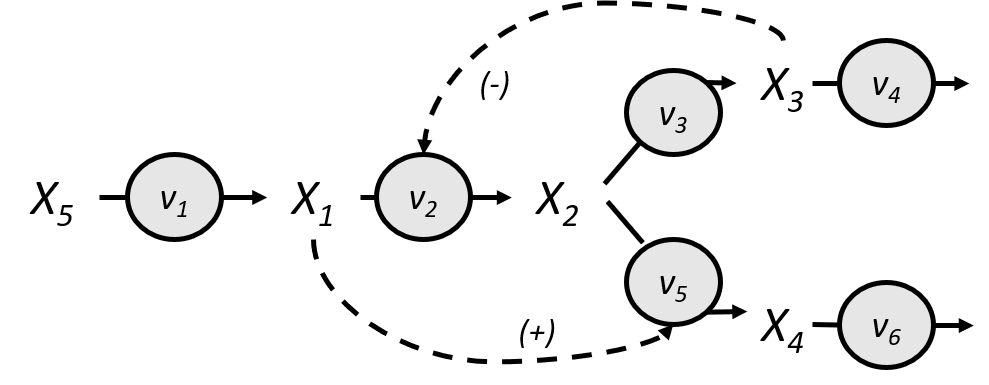}
\caption{An example of Hill-type metabolic network model \cite{csl:6}. }
\label{f2}
\end{figure}

 The corresponding embedded representation of the metabolic network, with $X_5$ as an independent variable, is as follows: 

\begin{multicols}{2}
\noindent$R_1: 0\rightarrow X_1$\\
$R_2: X_1+X_3 \rightarrow X_3+X_2$\\
$R_3: X_2 \rightarrow X_3$\\
$R_4: X_1+X_2 \rightarrow  X_1+X_4$\\ 
$R_5: X_3 \rightarrow 0$\\
$R_6: X_4 \rightarrow 0$\\
\end{multicols}

We apply the MILP algorithm on the SC Formalism approximation by  \cite{csl:6}  of the reference model depicted in Figure \ref{f2}. Using the framework,  the corresponding system of ODEs for the reference model is given as:

\begin{equation} 
\begin{split}
\frac{dX_1}{dt} = &V_1-\frac{V_2X^{n_{21}}_1X^{n_{23}}_3}{(k_{21}+X^{n_{21}}_1)(k_{23}+X^{n_{23}}_3)}\\
\frac{dX_2}{dt} = &\frac{V_2X^{n_{21}}_1X^{n_{23}}_3}{(k_{21}+X^{n_{21}}_1)(k_{23}+X^{n_{23}}_3)}-\frac{V_3X^{n_{32}}_2}{k_{32}+X^{n_{32}}_2}-\frac{V_4X^{n_{41}}_1X^{n_{42}}_2}{(k_{41}+X^{n_{41}}_1)(k_{42}+X^{n_{42}}_2)}\\
\frac{dX_3}{dt} = &\frac{V_3X^{n_{32}}_2}{k_{32}+X^{n_{32}}_2}-\frac{V_5X^{n_{53}}_3}{k_{53}+X^{n_{53}}_3}\\
\frac{dX_4}{dt} = &\frac{V_4X^{n_{41}}_1X^{n_{42}}_2}{(k_{41}+X^{n_{41}}_1)(k_{42}+X^{n_{42}}_2)}-\frac{V_6X^{n_{64}}_4}{k_{64}+X^{n_{64}}_4}
\end{split}
\end{equation}

\noindent where $V_1=8$, $V_2=84.2175$, $V_3=8$, $V_4=115.341$, $V_5=8$, and $V_6=8$. The interaction parameter matrix (containing the kinetic orders and dissociation constants) for the given system is: 

\[
  \begin{bmatrix}
0		&0		&0		&0		&0		&0		&0		&0\\
n_{21}	&0	 	&n_{23}	&0		&k_{21}	&0		&k_{32}	&0\\
0	 	&n_{32}	&0		&0		&0		&k_{32}	&0		&0\\
n_{41}	&n_{42}	&0		&0		&k_{41}	&k_{42}	&0		&0\\
0		&0	 	&n_{53}	&0		&0		&0		&k_{53}	&0\\
0		&0		&0	 	&n_{64}	&0		&0		&0		&k_{64}

  \end{bmatrix}
\]  
\noindent with $n_{21}=1$, $n_{23}=-0.8429$, $n_{32}=1$, $n_{41}=2.9460$, $n_{42}=3$, $n_{53}=1$,  $n_{64}=1$, $k_{21}=0.6705$, $k_{41}=0.8581$, $k_{42}=44.7121$, $k_{53}=1$,  and $k_{64}=1$.

Using the parameter values $u_{ij}=20,i,j=1,2,...,9$ for $i\ne j$ and $\epsilon=0.1$ and considering the same matrices $Y$ and $M$,  the sparse linearly conjugate network of the Hill-type system is \\
\begin{multicols}{2}
\noindent$R^{sparse2}_1: 0\rightarrow X_1$\\
$R^{sparse2}_2: X_1+X_3 \rightarrow X_3+X_2$\\
$R^{sparse2}_3: X_2 \rightarrow X_3$\\
$R^{sparse2}_4: X_1+X_2 \rightarrow  X_1+X_4$\\ 
$R^{sparse2}_5: X_3 \rightarrow 0$\\
$R^{sparse2}_6: X_4 \rightarrow 0$\\
\end{multicols}
\noindent with the corresponding system of ODEs

\begin{equation} 
\begin{split}
\frac{dX_1}{dt} = &\bar{V}_1-\frac{\bar{V}_2X^{n_{21}}_1X^{n_{23}}_3}{(k_{21}+{(c_1X_1)}^{n_{21}})(k_{23}+{(c_3X_3)}^{n_{23}})} \\
\frac{dX_2}{dt} = &\frac{\bar{V}_2X^{n_{21}}_1X^{n_{23}}_3}{(k_{21}+(c_1X_1)^{n_{21}})(k_{23}+(c_3X_3)^{n_{23}})}-\frac{\bar{V}_3X^{n_{32}}_2}{k_{32}+(c_2X_2)^{n_{32}}}-\frac{\bar{V}_4X^{n_{41}}_1X^{n_{42}}_2}{(k_{41}+(c_1X_1)^{n_{41}})(k_{42}+(c_2X_2)^{n_{42}})}\\
\frac{dX_3}{dt} = &\frac{\bar{V}_3X^{n_{32}}_2}{k_{32}+(c_2X_2)^{n_{32}}}-\frac{\bar{V}_5X^{n_{53}}_3}{k_{53}+(c_3X_3)^{n_{53}}}\\
\frac{dX_4}{dt} = &\frac{\bar{V}_4X^{n_{41}}_1X^{n_{42}}_2}{(k_{41}+(c_1X_1)^{n_{41}})(k_{42}+^(c_2X_2)^{n_{42}})}-\frac{\bar{V}_5X^{n_{54}}_4}{k_{54}+(c_4X_4)^{n_{54}}}.
\end{split}
\end{equation}

\noindent where $\bar{V}_1=0.8$, $\bar{V}_2=12.0921$, $\bar{V}_3=8$, $\bar{V}_4=10185531.88$, $\bar{V}_5=8$, and  $\bar{V}_6=8$.

The linearly conjugate dense realization was also obtained. The structure of the network is given as: 
\begin{multicols}{2}
\noindent$R^{dense2}_1: 0\rightarrow X_1$\\
$R^{dense2}_2: X_1+X_3 \rightarrow X_3+X_2$\\
$R^{dense2}_3: X_1+X_3 \rightarrow X_3$\\
$R^{dense2}_4: X_2 \rightarrow 0$\\
$R^{dense2}_5: X_2 \rightarrow X_3+X_2$\\
$R^{dense2}_6: X_2 \rightarrow X_3$\\
$R^{dense2}_7: X_3\rightarrow  0$\\ 
$R^{dense2}_8: X_1+X_2 \rightarrow  X_1$\\ 
$R^{dense2}_9: X_1+X_2 \rightarrow  X_1+X_4$\\ 
$R^{dense2}_{10}: X_4 \rightarrow 0$\\\
\end{multicols}

\noindent The conjugacy constants of the derived network are: $c_1=2.9555
$, $c_2=9.9140$, $c_3=0.4$, and $c_4=10$.  Using these constants and the computed $A_b$, we obtained the corresponding Kirchhoff matrix for the network: 

\[
A^{dense2}_k=
  \begin{bmatrix}
-2.7068	&0	&0		&0	&1.840		&0	&0		&0	&3.062\\
2.7068	&0	&0		&0	&0			&0	&320.697	&0	&0\\
0		&0	&-85.470	&0	&0			&0	&0		&0	&0\\
0		&0	&25.480	&0	&52.067		&0	&0		&0	&0\\
0		&0	&0		&0	&-54.168		&0	&0		&0	&0\\
0		&0	&59.990	&0	&0.260		&0	&0		&0	&0\\
0		&0	&0		&0	&0			&0	&-37310.233	 &0	&0\\
0		&0	&0		&0	&0			&0	&36989.536	 &0	&0\\
0		&0	&0		&0	&0			&0	&0		 &0	&-3.062

  \end{bmatrix}
.
\]

\noindent The associated ODEs for the dense realization is

\begin{equation}
\begin{split} 
\frac{dX_1}{dt} = & \bar{W}_1-\frac{\bar{W}_2X^{n_{21}}_1X^{n_{23}}_3}{(k_{21}+(c_1X_1)^{n_{21}})(k_{23}+(c_3X_3)^{n_{23}})}-\frac{\bar{W}_3X^{n_{31}}_1X^{n_{33}}_3}{(k_{31}+(c_1X_1)^{n_{31}})(k_{33}+(c_3X_3)^{n_{33}})}\\
\frac{dX_2}{dt} = & \frac{\bar{W}_2X^{n_{21}}_1X^{n_{23}}_3}{(k_{21}+(c_1X_1)^{n_{21}})(k_{23}+(c_3X_3)^{n_{23}})}- \frac{\bar{W}_4X^{n_{42}}_2}{(k_{42}+(c_2X_2)^{n_{42}})}-\frac{\bar{W}_6X^{n_{62}}_2}{(k_{62}+(c_2X_2)^{n_{62}})}-\\
&\frac{\bar{W}_8X^{n_{81}}_1X^{n_{82}}_2}{(k_{81}+(c_1X_1)^{n_{81}})(k_{82}+(c_2X_2)^{n_{82}})}- \frac{\bar{W}_9X^{n_{91}}_1X^{n_{92}}_2}{(k_{91}+(c_1X_1)^{n_{91}})(k_{92}+(c_2X_2)^{n_{92}})}\\
\frac{dX_3}{dt} = & \frac{\bar{W}_5X^{n_{52}}_2}{(k_{52}+(c_2X_2)^{n_{52}})}+\frac{\bar{W}_6X^{n_{62}}_2}{(k_{62}+(c_2X_2)^{n_{62}})}-\frac{\bar{W}_7X^{n_{73}}_3}{(k_{73}+(c_3X_3)^{n_{73}})}\\
\frac{dX_4}{dt} = & \frac{\bar{W}_9X^{n_{91}}_1X^{n_{92}}_2}{(k_{91}+(c_1X_1)^{n_{91}})(k_{92}+(c_2X_2)^{n_{92}})}-\frac{\bar{W}_{10}X^{n_{104}}_4}{(k_{104}+(c_4X_4)^{n_{104}})}
\end{split}
\end{equation}

\noindent with $\bar{W}_1=2.7068$, $\bar{W}_2=54.3516$, $\bar{W}_3=127.9649$, $\bar{W}_4=7.0009$, $\bar{W}_5=52.067$, $\bar{W}_6=0.9914$, $\bar{W}_7=8$, $\bar{W}_8=2372.74$, $\bar{W}_9=273674.2$, and $\bar{W}_{10}=8$. The kinetic orders and dissociation constants are $n_{21}=n_{31}=1$, $n_{23}=n_{33}=-0.8429$, $n_{42}=n_{52}=n_{62}=1$, $n_{73}=1$, $n_{81}=n_{91}=2.9460$, $n_{82}=n_{92}=3$, and $n_{104}=1$, $k_{21}=k_{31}=0.6705$, $k_{23}=k_{33}=3.9065$, $k_{42}=k_{52}=k_{62}=1$, $k_{73}=1$, $k_{81}=k_{91}=0.8581$, $k_{82}=k_{92}=44.7121$, and $k_{104}=1$.\\

The linearly conjugate sparse network has also 6 reactions which is equal to the number of reactions of the derived linearly conjugate sparse system with power-law kinetics. Whereas, the dense realization of the SC model has 10 reactions. The graphs of the individual trajectories of the original Hill-type system and the linearly conjugate systems  are depicted in Figures~\ref{scsparse:C1}-\ref{scsparse:C4} and Figures~\ref{scdense:D1}-\ref{scdense:D4}, respectively.

%%%%%%%%%%%%%%%%%%%%%%%%%%%%%%% 
%%%%%%%%%%%%%%%%%%%%%%%%%%%%%%%
%%%%%%%%%%%%%%%%%%%%%%%%%%%%%%%
\section{Conclusion}
Different networks could generate the same set of ODEs making them dynamically equivalent. In the past few years, various authors have pioneered the use of MILP algorithms for determining linear conjugacy between MAK systems \cite{csl:19,csl:16a,csl:16,csl:16b}, between rational functions systems \cite{csl:19a}, between GMAK systems \cite{csl:16c} and between PL-RDK systems \cite{csl:8}. In the work of \cite{csl:8}, they extended the JSC for linear conjugacy from MAK systems to PL-RDK systems. It is limited to power law kinetic systems with branching reactant complexes that have identical kinetic orders. In this study, we further extended the algorithm for branching reactant complexes with different kinetic orders.

We summarize below main results presented in this paper:
\begin{enumerate}

\item We showed that any non-complex factorizable (NF) RID kinetic system can be dynamically equivalent to a CF system via CF-transformation (Theorem 1).  

\item We further illustrated the usefulness of CF-RMA through the extended proof of Subspace Coincidence Theorem for the kinetic and stoichiometric subspaces (KSSC) of NF kinetic systems. 

\item We extended the JSC for linear conjugacy to the CF subset of RID kinetic systems, i.e., those whose interaction map $I_K: \Omega \rightarrow \mathbb{R}^\mathscr{R}$ factorizes via the space of complexes $\mathbb{R}^\mathscr{C}$:  $I_K = I_k\circ \psi_K$ with $\psi_K:  \Omega \rightarrow \mathbb{R}^\mathscr{C}$ as factor map and $I_k = diag(k)\circ \rho'$   with  $\rho': \mathbb{R}^\mathscr{C} \rightarrow \mathbb{R}^\mathscr{R}$ assigning the value at a reactant complex to all its reactions (Theorem 4).

\item We demonstrated (with running examples: Examples 2 - 4) that linear conjugacy can be generated for any RID kinetic systems by applying the JSC to any NF kinetic system that are transformed to CF kinetic system. The extended JSC for linear conjugacy to CF-RID systems is combined with the CF-RM method to provide the general computational solution to construct linear conjugates of any RID system.  

\item For a large subset of RID kinetic systems RIPK, which have interaction parameter maps, we illustrated how the proposed approach of this paper can also be applied and that the computational solution is always feasible. We presented an example of HTK which was also known as SC Formalism.

\end{enumerate}

%%%%%%%%%%%%%%%%%%%%%%%%%%%%%%%%
%%%%%%%%%%%%%%%%%%%%%%%%%%%%%%%%%
%%%%%%%%%%%%%%%%%%%%%%%%%%%%%%%%% 
%\section{Results}
%The body text is in 12 point normal Times New Roman, 
%the line space is at least 15 point.

%\section{Discussion}

%\section{Conclusion}

\section*{Acknowledgments}
We thank Casian Pantea for presenting the idea of transforming any power law kinetic system to a dynamically equivalent reactant-determined system, which was the basis for the development of the CF-RM method. ARL held research fellowships from De La Salle University and would like to acknowledge the support of De La Salle University's Research Coordination Office.

\section*{Conflict of interest}
We have no conflicts of interest to disclose.

%For more questions regarding reference style, please refer to the \href{http://www.ncbi.nlm.nih.gov/books/NBK7256/}{Citing Medicine}.
\newpage
\section*{Supplementary Materials}
%\beginsupplement

\begin{table}[H]
\caption{List of abbreviations}
\small
\center
\begin{tabular}{|l|l|}
	\hline
	\textbf{Abbreviations} & \textbf{Meaning}\\
	\hline
	\hline
	CF & Complex Factorizable \\
	%\hline
	%CFK & Complex Factorizable Kinetics\\
	\hline
	CFM & maximal CF-subsystem\\
	\hline
	CFS & CF-subsets\\
	\hline
	CF-RM & Complex Factorization by Reactant Multiples \\
	\hline
	CKS & Chemical Kinetic System \\
	\hline
	CRN & Chemical Reaction Network \\
	\hline
	CRNT & Chemical Reaction Network Theory \\
	\hline
	FSS & Factor Span Surjective \\
	\hline
	GMAK & Generalized Mass Action Kinetics \\
	\hline
	HTK & Hill-Type Kinetics \\
	\hline
	JSC & Johnston-Siegel Criterion \\
	\hline
	KSSC & Kinetic and Stoichiometric Subspace Coincidence \\
	\hline
	LRD & Low Reactant Deficiency \\
	\hline
	MAK & Mass Action Kinetics \\
	\hline
	MILP & Mixed Integer Linear Programming \\
	\hline
	NDK & Non-reactant Determined Kinetics\\
	\hline
	NF & Non-complex Factorizable \\
	\hline
	ODE & Ordinary Differential Equation \\
	\hline
	PLK & Power Law Kinetics \\
	\hline
	PL-RDK & Power Law - Reactant Determined Kinetics \\
	\hline
	PL-TIK & $\bf{\hat{T}}-$rank maximal Kinetics\\
	\hline
	PT & Point Terminal \\
	\hline
	RDK & Reactant Determined Kinetics \\
	\hline
	RID & Rate constant Interaction map Decomposable \\
	\hline
	RIDK & Rate constant Interaction map Decomposable Kinetics \\
	\hline
	RIPK & RID kinetics with Intersection Parameters map\\
	\hline
	RIP-CFK & CF RIPK\\
         \hline
	RIP-NFK & NF RIPK\\
	\hline
	SC & Saturation-Cooperativity\\
	\hline
	SFRF & Species Formation Rate Function \\
	\hline
	SRD & Sufficient Reactant Deficiency \\
	\hline
	TBD & Terminality Bounded by Deficiency \\
	\hline
	TND & Terminality Not Bounded by Deficiency \\
	\hline
	\end{tabular}
	\end {table}

        \begin{table}[H]
        \caption{List of symbols}
        \small
	\center
	\begin{tabular}{|l|l|}
	\hline
	\textbf{List of Symbols} & \textbf{Meaning} \\
	\hline
	\hline
	$\mathbb{R}^{\mathscr{C}}$ & complex vector space \\
	\hline
	$( \mathscr{N}^{\ast}, K^{\ast} )$ & CRN of CF-transform of an NF system \\
	\hline
	$\delta$ & deficiency of a CRN \\
	\hline
	$\psi_K$ & factor map \\
	\hline
	$I_a$ & incidence mapping \\
	\hline
	$I_K$ & interaction mapping \\
	\hline
	$A_k$ & $k$-Laplacian map \\
	\hline
	$F$ & kinetic order matrix \\
	\hline
	$Y$ & matrix of complexes \\
	\hline
	$N_R$ & number or CF-subsets \\
	\hline
	$n$ & number of complexes \\
	\hline
	$l$ & number of linkage classes \\
	\hline
	$l^{\ast}$ & number of linkage classes of $( \mathscr{N}^{\ast}, K^{\ast} )$ \\
	\hline
	$l_{b}^{\ast}$ & number of new linkage classes due to link-breaking \\
	\hline
	$l_{r}^{\ast}$ & number of new linkage classes by new reactants \\
	\hline
	$n_r$ & number of reactants \\
	\hline
	$r$ & number of reactions \\
	\hline
	$m$ & number of species \\
	\hline
	$sl$ & number of strong linkage classes \\
	\hline
	$t$ & number of terminal linkage classes \\
	\hline
	$\pi$ & product mapping \\
	\hline
	$s$ & rank of the CRN \\
	\hline
	$\delta_\rho$ & reactant deficiency \\
	\hline
	$\rho$ & reactant mapping \\
	\hline
	$q$ & reactant rank \\
	\hline
	$R$ & reactant subspace \\
	\hline
	$\mathbb{R}^{\mathscr{R}}$ & reaction vector space \\
	\hline
	$\mathscr{R}(y)$ & set of branching reactions \\
	\hline
	$\mathscr{C}$ & set of complexes \\
	\hline
	$\mathscr{L}$ & set of linkage classes \\
	\hline
	$E_+ ( \mathscr{N}, K )$ & set of positive equilibria of CKS \\
	\hline
	$\mathscr{R}$ & set of reactions \\
	\hline
	$\mathscr{S}$ & set of species \\
	\hline
	$\mathbb{R}^{\mathscr{S}}$ & species vector space \\
	\hline
	$S$ & stoichiometric subspace of CRN \\
	\hline
	\end{tabular}
	\end{table}

%\begin{table}[H]
%\caption{Parameters for the GMA model}
%\small
%\center
%\begin{tabular}{c c}
%\hline
%Parameter& Value\\ 
%\hline
%$\gamma_1$ & 8\\
%\hline
%$\gamma_2$ & 10.3374\\
%\hline
%$\gamma_3$ & 4.0556\\
%\hline
%$\gamma_4$&1.3897\\
%\hline
%$\gamma_5$&4.0556\\
%\hline
%$\gamma_6$&4.0556\\
%\hline
%$f_{21}$&0.4130\\
%\hline
%$f_{23}$&-0.7065\\
%\hline
%$f_{32}$&0.4171\\
%\hline
%$f_{41}$&1.4646\\
%\hline
%$f_{42}$&2.8274\\
%\hline
%$f_{53}$&0.4171\\
%\hline
%$f_{64}$&0.5829\\
%\hline
%\end{tabular}
%\label{t2}
%\end{table}

\begin{figure}[H]
\center
\subfigure[]{
\resizebox*{3cm}{!}{\includegraphics{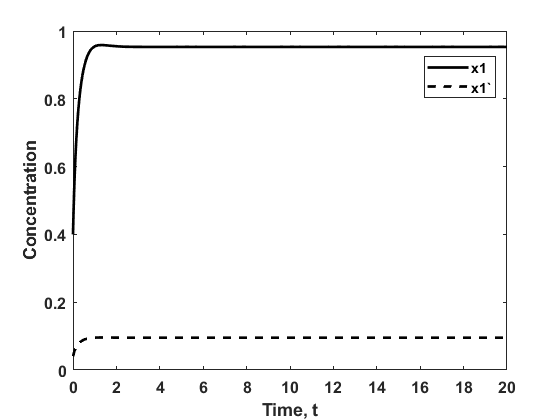}}}\hspace{1pt}
\subfigure[]{
\resizebox*{3cm}{!}{\includegraphics{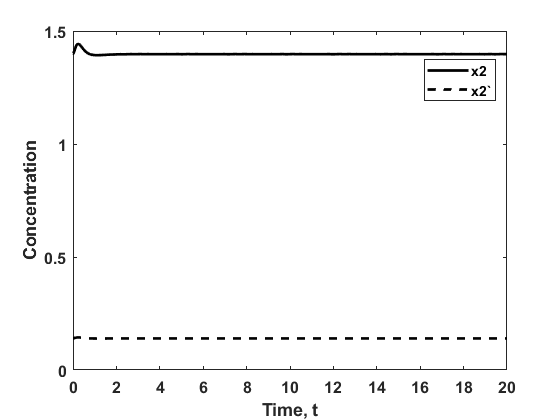}}}\hspace{1pt}
\subfigure[]{
\resizebox*{3cm}{!}{\includegraphics{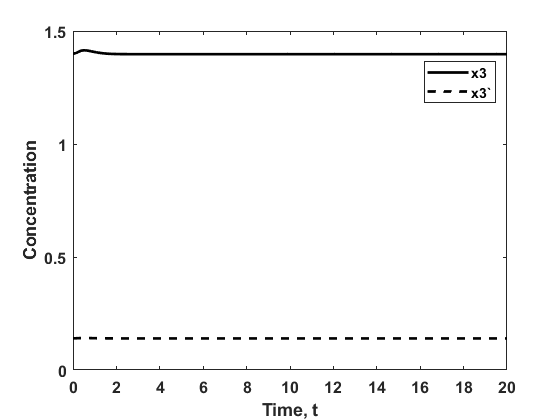}}}\hspace{1pt}
\subfigure[]{
\resizebox*{3cm}{!}{\includegraphics{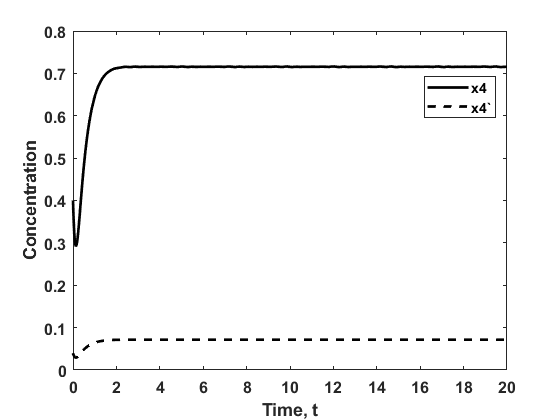}}}\hspace{1pt}
\subfigure[]{
\resizebox*{3cm}{!}{\includegraphics{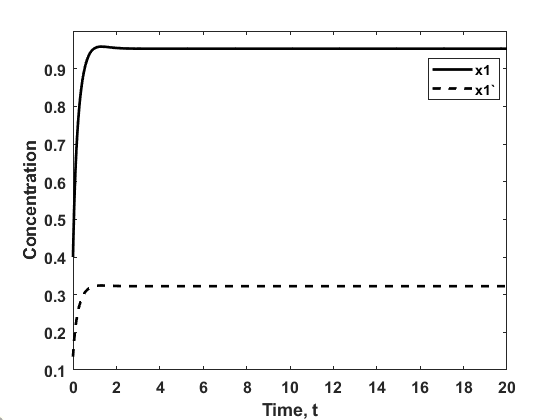}}}\hspace{1pt}
\subfigure[]{
\resizebox*{3cm}{!}{\includegraphics{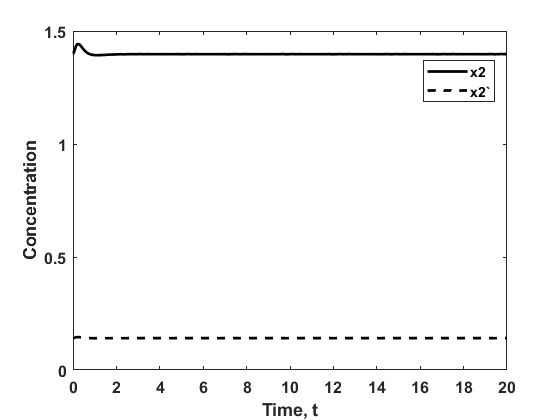}}}\hspace{1pt}
\subfigure[]{
\resizebox*{3cm}{!}{\includegraphics{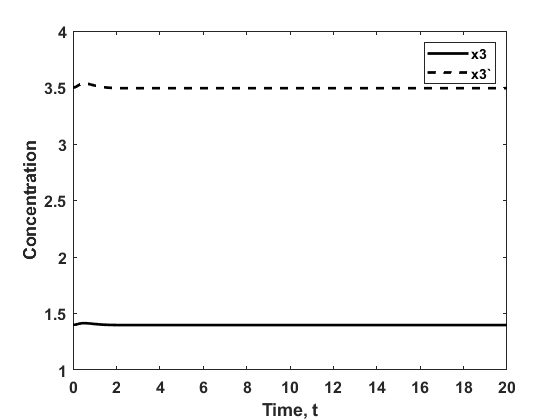}}}\hspace{1pt}
\subfigure[]{
\resizebox*{3cm}{!}{\includegraphics{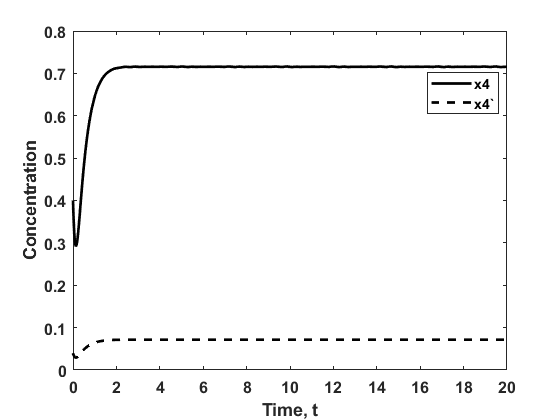}}}
\caption{The graphs of the trajectories for the original hill-type model and a linearly conjugate of sparse (first row) and dense (second row) realization. $X_i'$ represents a trajectory in the sparse realization (shown in first row) and dense realization(shown in second row).}
\label{scsparsedense}
\end{figure}

\end{document}